%% file: manuscript_revised_2.tex
\documentclass[11pt]{article}
\usepackage[margin=1in]{geometry}
\usepackage[T1]{fontenc}
\usepackage{lmodern,microtype}
\usepackage{amsmath,amssymb,amsthm,mathtools}
\usepackage{enumitem,booktabs,graphicx}
\usepackage{placeins}
\usepackage{needspace}

\usepackage[numbers,sort&compress]{natbib}
\usepackage[hidelinks]{hyperref}
\usepackage{xcolor}
\usepackage{etoolbox}
\usepackage{authblk}
\usepackage{orcidlink}

\usepackage{todonotes}

\definecolor{revision}{RGB}{25,65,155}

\newcommand{\defeq}{\mathrel{\coloneqq}}
\hypersetup{pdftitle={Relative Primal-Dual Gap Certificates for Operator-Composite Trust-Region Methods},
pdfauthor={Harbir Antil},
pdfsubject={Revised manuscript with four-lobe control experiments},
pdfkeywords={nonsmooth optimization, trust regions, primal-dual gap, operator composition}}
\numberwithin{equation}{section}
\newtheorem{theorem}{Theorem}[section]
\newtheorem{lemma}[theorem]{Lemma}
\newtheorem{proposition}[theorem]{Proposition}
\newtheorem{corollary}[theorem]{Corollary}
\theoremstyle{definition}
\newtheorem{definition}[theorem]{Definition}
\newtheorem{assumption}[theorem]{Assumption}

\newtheorem{algorithm}[theorem]{Algorithm}
\theoremstyle{remark}
\newtheorem{remark}[theorem]{Remark}
\DeclareMathOperator{\dom}{dom}
\DeclareMathOperator{\prox}{prox}

\DeclareMathOperator{\pred}{pred}

\newcommand{\R}{\mathbb R}
\newcommand{\pair}[2]{\langle #1,#2\rangle}
\newcommand{\norm}[1]{\lVert #1\rVert}
\newcommand{\ip}[2]{(#1,#2)_W}
\newcommand{\RW}{\mathcal R_W}
\newcommand{\dd}{\,\mathrm d}
\setlist[enumerate]{leftmargin=*,itemsep=3pt,topsep=4pt}
\setlist[itemize]{leftmargin=*,itemsep=3pt,topsep=4pt}

\usepackage{titlesec}
\titleformat{\section}{\normalfont\scshape\centering}{\thesection.}{0.5em}{}
\titleformat*{\subsection}{\itshape}
\titleformat*{\subsubsection}{\itshape}

 \providecommand{\keywords}[1]
 {
 	{\small\emph{Keywords:} #1}
 }

\title{\vspace{-10mm}\bf Relative Primal--Dual Gap Certificates\texorpdfstring{\\}{ }for Operator-Composite Trust-Region Methods
\thanks{This work is partially supported by the Office of Naval Research under Award No. N00014-24-1-2147, the National
Science Foundation under Grant DMS-2408877, and the Air Force Office of Scientific Research under Award No. FA9550-22-
1-0248}}
\author[1]{Harbir Antil\,\orcidlink{0000-0002-6641-1449}%
	\thanks{Email: \url{hantil@gmu.edu}}}
\date{}	
	\affil[1]{\small{Department of Mathematical Sciences and the Center for Mathematics and Artificial Intelligence (CMAI), George Mason University, Fairfax, VA 22030, USA}}

\begin{document}
\maketitle

\vspace{-1cm}

\begin{abstract}
We study trust-region minimization of a smooth, possibly nonconvex
functional plus a convex functional composed with a bounded linear
operator. A relative primal--dual gap condition controls both the error
in an approximate proximal-gradient step and its linear-model decrease.
Together with a computable absolute stationarity test, it yields a finite
Cauchy search, convergence of the proximal stationarity measure to zero,
and an $O(\varepsilon^{-2})$ bound on outer trials. The outer analysis
allows the linear operator to take values in a Banach space and does
not require dual attainment. When the operator takes values in a Hilbert space and the regularizer
is finite and Lipschitz, the dual proximal-gradient method produces finite
gaps tending to zero, provided the required proximal maps and functional
values can be evaluated. We prove $O(j^{-1})$ gap bounds for both
recovered and averaged primal candidates and give a sharper bound 
on the primal error for exactly recovered points. A semilinear elliptic 
control problem with unsmoothed total-variation regularization and 
an $L^2$ control cost illustrates the method in the full $H^1$ metric.  
Across five meshes, outer and state Newton counts remain constant,
while interior-point iteration counts vary mildly.
\end{abstract}

\section{Introduction}
\label{sec:introduction}
Let $W$ be a real Hilbert space, let $X$ be a real Banach space, and let
$T:W\to X$ be bounded and linear. We consider
\begin{equation}\label{eq:problem}
 \min_{w\in W} F(w)\defeq f(w)+R(Tw),
\end{equation}
where $f$ is smooth and $R$ is proper, lower semicontinuous, and convex.
The smooth term need not be convex. Such problems arise in inverse problems
and partial differential equation (PDE) constrained optimization when a regularizer
acts on derivatives, observations, or groups of control variables. The
factorization matters computationally: the proximal map of $R$ or its
conjugate can be simple even when the proximal map of $R\circ T$ is expensive.

A proximal trust-region method uses a proximal-gradient step to obtain
sufficient model decrease. When this step is computed
iteratively, a stopping condition must guarantee the decrease needed by the
outer method. A small difference between consecutive inner iterates does
not by itself provide this guarantee. Moreover, for an extended-valued
regularizer, a point close to its effective domain need not belong to that
domain. A finite primal--dual gap must therefore be evaluated at a point
where the original proximal objective is finite.

This work is motivated by a posteriori error analysis based on duality
for convex variational problems, developed by Bartels and 
Kaltenbach and in joint work with the author and coauthors
\citep{BartelsKaltenbach2023,BartelsKaltenbach2024,AntilBartelsKaltenbachKhandelwal2025,
AntilKaltenbachKirk2026}. In that framework, the primal--dual gap is a
computable measure of the combined primal and dual errors, valid for
admissible approximations independently of how they were computed.
Here we apply this principle to the strongly convex proximal subproblem
inside a possibly nonconvex trust-region method. The outer objective need
not satisfy the convex duality relations used for this subproblem.

We use a relative primal--dual gap condition. At a current point $w$, write
$\nabla f(w)$ for the Hilbert gradient, choose $t>0$, and set
$q\defeq w-t\nabla f(w)$. For the proximal objective
$\Psi_{q,t}(y)\defeq\norm{y-q}_W^2/(2t)+R(Ty)$, a primal candidate $y=w+s$
and a dual candidate $\mu$ determine a gap $\mathcal G_{q,t}(y,\mu)$.
The condition
\begin{equation}\label{eq:relative_intro}
 \mathcal G_{q,t}(w+s,\mu)
 \le \frac{\sigma^2}{2t}\norm{s}_W^2,
 \qquad 0<\sigma<1,
\end{equation}
gives upper and lower bounds on the computed step in terms of the
exact step, and it guarantees decrease of the linear composite
model. Both conclusions are needed in the trust-region analysis.

Trust-region methods and their Cauchy decrease estimates are classical
\citep{ConnGouldToint2000}. Baraldi and Kouri
\citep{BaraldiKouri2023} allow inexact smooth evaluations in a proximal
trust-region method; Maia, Baraldi, and Kouri
\citep{MaiaBaraldiKouri2026} also permit inexact weighted proximal points.
Aravkin, Baraldi, and Orban \citep{AravkinBaraldiOrban2022} analyze
proximal quasi-Newton trust-region methods. A linear composition can be
included as the abstract convex term in these formulations.
Kouri \citep{Kouri2026} treats a broader structured, potentially nonlinear
composition and develops a dual inner method, while the corresponding
outer convergence analysis assumes exact composite proximal evaluations.

Relative proximal errors and dual gap conditions likewise predate this
work \citep{SolodovSvaiter1999,SalzoVilla2012,RaschChambolle2020}.
Villa, Salzo, Baldassarre, and Verri
\citep[Proposition~2.2 and Theorem~6.1]{VillaSalzoBaldassarreVerri2013}
derive error bounds from primal--dual gaps and recover primal points
from dual iterates for linear compositions. Their results also show that
the gaps of recovered primal points tend to zero for regularizers that
are finite everywhere. Bonettini and coauthors \citep{BonettiniEtAl2017}
use relative dual tests for composite proximal subproblems in nonconvex
line-search methods, and Lee and Wright \citep{LeeWright2019} analyze
inexact quadratic approximation. The a posteriori error identities explain our choice of certificate;
the inexact proximal literature provides closely related approximation
conditions and methods for constructing the candidate pairs. Both
connections rely on Fenchel duality.

We use computable primal--dual gap bounds to construct trust-region
steps. The relative condition gives the bounds on step norms and the
model decrease required for a finite Cauchy search. A separate absolute
test certifies that the stationarity tolerance is met, including at
stationary points. Together, these estimates prove convergence of the
stationarity measure to zero and the standard first-order bound on outer
trials for bounded, possibly indefinite model operators. The analysis 
does not require exact minimization of the trust-region model. 
For Lipschitz regularizers satisfying the evaluation assumptions
below, we prove gap estimates for recovered and averaged primal points
in Hilbert space.
We verify these assumptions for a semilinear elliptic control
problem with total-variation regularization and $L^2$ control cost, and 
prove existence of an optimal $H^1$ control on convex domains.

\paragraph{Scope.}
The outer theorem assumes that each requested proximal problem admits
an inner procedure whose gaps are finite, can be evaluated, and tend to zero.
We verify this assumption for finite Lipschitz regularizers with the stated
proximal and functional evaluations. For an indicator regularizer, a finite
gap requires feasibility of the tested primal point; the present method
does not allow approximate feasibility in that test. The analysis
assumes exact smooth evaluations and linear operations.

\section{Problem setting}
\label{sec:setting}
Let $X$ be a real Banach space and let $X^*$ be its continuous dual. The
Riesz map $\RW:W\to W^*$ is defined by
$\pair{\RW u}{v}=\ip{u}{v}$. The dual norm on $W^*$ is
$\norm{\ell}_{W^*}=\norm{\RW^{-1}\ell}_W$. We write
$T^*:X^*\to W^*$ for the Banach adjoint. Thus no Riesz identification is
made in $X$ unless it is explicitly assumed to be Hilbert.
For a proper convex functional $h$ on a Banach space, we use
$\ell\in\partial h(v)$ to mean that $h(v)$ is finite and
$h(z)\ge h(v)+\pair{\ell}{z-v}$ for every $z$.

\begin{assumption}\label{ass:basic}
The functional $R:X\to(-\infty,+\infty]$ is proper, convex, and norm lower
semicontinuous, and $\dom(R\circ T)\ne\emptyset$. The function $f:W\to\R$
is continuously Fr\'echet differentiable, and its Hilbert gradient $\nabla f$ is
globally Lipschitz with constant $L_f\ge0$. The objective $F$ is bounded
below on $\dom(R\circ T)$ by a finite number $F_{\rm low}$.
\end{assumption}

The gradient convention is
$\ip{\nabla f(w)}{v}=f'(w)[v]$ for $v\in W$, so
$\RW\nabla f(w)=f'(w)$. Global Lipschitz continuity supplies a
uniform bound for all segments and pairs used below.
Write $\varphi\defeq R\circ T$ when the factorization need not be displayed. This is proper,
convex, and lower semicontinuous. For $w\in\dom\varphi$ and $t>0$, define the 
proximal-gradient step
\begin{equation}\label{eq:exact_prox}
 p(t)=p(w;t)\defeq \arg\min_{s\in W}
 \left\{\ip{\nabla f(w)}{s}+\frac1{2t}\norm{s}_W^2+\varphi(w+s)\right\}.
\end{equation}

The convex-analytic facts used below are standard; see
\citet{BauschkeCombettes2017,EkelandTemam1999}. We recall these properties in our Hilbert-space notation.
\begin{lemma}[Properties of the proximal-gradient step]\label{lem:prox}
The step in \eqref{eq:exact_prox} exists uniquely and satisfies
$-f'(w)-t^{-1}\RW p(t)\in\partial\varphi(w+p(t))$.
For $0<t_1<t_2$,
\begin{equation}\label{eq:path_norms}
 \norm{p(t_1)}_W\le\norm{p(t_2)}_W,
 \qquad \frac{\norm{p(t_1)}_W}{t_1}\ge\frac{\norm{p(t_2)}_W}{t_2},
 \qquad \lim_{t\downarrow0}\norm{p(t)}_W=0.
\end{equation}
Moreover, $p(t)=0$ for one, equivalently every, $t>0$ if and only if
$-f'(w)\in\partial\varphi(w)$.
\end{lemma}
\begin{proof}
For fixed $w$, the function
$h_w(s)\defeq\varphi(w+s)+\ip{\nabla f(w)}{s}$
is proper, lower semicontinuous, and convex.
Existence, uniqueness, and the optimality inclusion follow from
\citet[Propositions~12.15 and~16.44]{BauschkeCombettes2017}.
The inclusion also gives the zero-step characterization.

The two norm comparisons follow from
\citet[Lemma~2]{BaraldiKouri2023}, with $x=w$ and
$d=-\nabla f(w)$. Finally, since $0\in\dom h_w$,
\citet[Proposition~12.33(iii)]{BauschkeCombettes2017},
applied to $h_w$ at the origin, gives
$t^{-1}\norm{p(t)}_W^2\to0$ as $t\downarrow0$.
In particular, $\norm{p(t)}_W\to0$.
\end{proof}

For a proper lower semicontinuous convex functional $h$ on a Hilbert
space, write $\prox_{th}(q)$ for the minimizer of
$h(y)+\norm{y-q}^2/(2t)$ in that space's norm. Thus
$w+p(w;t)=\prox_{t\varphi}(w-t\nabla f(w))$.
Fix a reference parameter $\tau>0$ and define
\begin{equation}\label{eq:criticality}
 \chi(w)\defeq \frac1\tau\norm{p(w;\tau)}_W.
\end{equation}
The convergence analysis uses $\chi$; the algorithm does not evaluate
it exactly.

Although we first defined $p(w;t)$ and $\chi(w)$ for
$w\in\dom\varphi$, the proximal formula defines them on all of $W$.
The following result records the continuity needed below.
\begin{lemma}[Continuity and stationarity]\label{lem:chi}
The function $\chi$ extends to a Lipschitz function on $W$, with constant
$L_\chi\defeq 2/\tau+L_f$. At a point in $W$, $\chi(w)=0$ if and only if
$w\in\dom\varphi$ and $0\in f'(w)+\partial\varphi(w)$.
\end{lemma}
\begin{proof}
For $z_i=\prox_{\tau\varphi}(q_i)$, monotonicity of $\partial\varphi$
gives $\norm{z_1-z_2}_W^2\le\ip{q_1-q_2}{z_1-z_2}$, hence the proximal
map is nonexpansive. Apply this with $q_i=w_i-\tau \nabla f(w_i)$ and subtract
$w_i$. The difference of the steps has norm at most
$(2+\tau L_f)\norm{w_1-w_2}_W$. The reverse triangle inequality proves
the Lipschitz estimate. A zero step is equivalent to
$w=\prox_{\tau\varphi}(w-\tau\nabla f(w))$;
proximal optimality is equivalent to
$w\in\dom\varphi$ and $-f'(w)\in\partial\varphi(w)$.
\end{proof}

\section{Primal--dual gap certificate}
\label{sec:certificate}
We express the gap as two nonnegative optimality defects and bound
the primal error without assuming dual attainment. When a zero-gap
dual solution exists, the resulting error identity specializes
\citet[Theorem~3.3]{BartelsKaltenbach2024} to the proximal objective.
For $q\in W$ and $t>0$, let
\begin{equation}\label{eq:primal_dual}
 \begin{aligned}
 \Psi_{q,t}(y)&\defeq\frac1{2t}\norm{y-q}_W^2+R(Ty),\\
 D_{q,t}(\mu)&\defeq\pair{\mu}{Tq}
       -\frac t2\norm{T^*\mu}_{W^*}^2-R^*(\mu).
 \end{aligned}
\end{equation}
The conjugate is $R^*(\mu)\defeq\sup_{x\in X}(\pair{\mu}{x}-R(x))$.
All gap tests below require $Ty\in\dom R$ and $\mu\in\dom R^*$.

For $w \in \dom\varphi$ and \(q=w-t\nabla f(w)\), the substitution \(y=w+s\)
identifies minimization of \(\Psi_{q,t}\) with
\eqref{eq:exact_prox}, up to the additive constant
\(t\|\nabla f(w)\|_W^2/2\). In particular, its minimizer
is \(w+p(w;t)\).

\begin{proposition}[Gap identity and error bound]\label{prop:gap}
Set $r\defeq t^{-1}\RW(y-q)+T^*\mu$ and
$\varepsilon_R\defeq R(Ty)+R^*(\mu)-\pair{\mu}{Ty}$. Then
\begin{equation}\label{eq:gap_identity}
 \mathcal G_{q,t}(y,\mu)\defeq \Psi_{q,t}(y)-D_{q,t}(\mu)
 =\varepsilon_R+\frac t2\norm{r}_{W^*}^2\ge0.
\end{equation}
The dual objective in \eqref{eq:primal_dual} is a lower bound for the primal
minimum. If $y_\star$ minimizes $\Psi_{q,t}$, then
\begin{equation}\label{eq:gap_distance}
 \frac1{2t}\norm{y-y_\star}_W^2
 \le \Psi_{q,t}(y)-\Psi_{q,t}(y_\star)
 \le\mathcal G_{q,t}(y,\mu).
\end{equation}
No dual attainment or chain-rule qualification is required for these claims.
\end{proposition}
\begin{proof}
Fenchel--Young gives $R(Ty)\ge\pair{\mu}{Ty}-R^*(\mu)$.
Minimizing the right side plus the quadratic over $y$ gives
$y=q-t\RW^{-1}T^*\mu$ and the value $D_{q,t}(\mu)$; this proves weak
duality. Expanding
$\frac t2\norm{t^{-1}\RW(y-q)+T^*\mu}_{W^*}^2$
gives the quadratic in $\Psi_{q,t}$, the quadratic in $-D_{q,t}$, and the
cross term $\pair{\mu}{T(y-q)}$. Adding $\varepsilon_R$ proves
\eqref{eq:gap_identity}. Strong convexity of $\Psi_{q,t}$ and
$0\in\partial\Psi_{q,t}(y_\star)$ give the first inequality in
\eqref{eq:gap_distance}; weak duality gives the second.
\end{proof}

\begin{proposition}[Solution errors represented by the gap]
\label{prop:error_identity}
Let $y_\star$ be the unique minimizer of $\Psi_{q,t}$, and let
$(y,\mu)$ have a finite gap. Then 
\begin{equation}\label{eq:two_distance_bound}
 \frac{\norm{y-y_\star}_W^2}{2t}
 +\frac{\norm{q-t\RW^{-1}T^*\mu-y_\star}_W^2}{2t}
 \le \mathcal G_{q,t}(y,\mu).
\end{equation}
In particular, if $y=q-t\RW^{-1}T^*\mu$, then
\begin{equation}\label{eq:recovered_gap_distance}
 \norm{y-y_\star}_W^2\le t\mathcal G_{q,t}(y,\mu).
\end{equation}
These inequalities do not require dual attainment.

If there is a dual vector $\mu_\star$ such that
$\mathcal G_{q,t}(y_\star,\mu_\star)=0$, then the following
exact identity holds:
\begin{equation}\label{eq:exact_error_identity}
\begin{aligned}
 \mathcal G_{q,t}(y,\mu)
 ={}&\frac{\norm{y-y_\star}_W^2}{2t}
      +\frac t2\norm{T^*(\mu-\mu_\star)}_{W^*}^2\\
 &+R(Ty)-R(Ty_\star)-\pair{\mu_\star}{T(y-y_\star)}\\
 &+R^*(\mu)-R^*(\mu_\star)-\pair{\mu-\mu_\star}{Ty_\star}.
\end{aligned}
\end{equation}
Each of the two remainders on the last two lines is nonnegative.
\end{proposition}
\begin{proof}
Split the gap as
\[
 \mathcal G_{q,t}(y,\mu)
 =\Psi_{q,t}(y)-\Psi_{q,t}(y_\star)
  +\mathcal G_{q,t}(y_\star,\mu).
\]
Strong convexity bounds the first term below by
$\norm{y-y_\star}_W^2/(2t)$; the residual term in
\eqref{eq:gap_identity} bounds the second below by
$\norm{y_\star-q+t\RW^{-1}T^*\mu}_W^2/(2t)$.
This proves \eqref{eq:two_distance_bound}, and the recovery formula
makes the two squared distances equal and gives \eqref{eq:recovered_gap_distance}.

If $\mathcal G_{q,t}(y_\star,\mu_\star)=0$, then
\eqref{eq:gap_identity} and Fenchel equality give
\[
 y_\star-q=-t\RW^{-1}T^*\mu_\star,
 \qquad \mu_\star\in\partial R(Ty_\star),\qquad
 R^*(\mu_\star)=\pair{\mu_\star}{Ty_\star}-R(Ty_\star).
\]
Expanding the two quadratic terms therefore yields
\[
\begin{aligned}
 \Psi_{q,t}(y)-\Psi_{q,t}(y_\star)
 &=\frac{\norm{y-y_\star}_W^2}{2t}
   +R(Ty)-R(Ty_\star)-\pair{\mu_\star}{T(y-y_\star)},\\
 D_{q,t}(\mu_\star)-D_{q,t}(\mu)
 &=\frac t2\norm{T^*(\mu-\mu_\star)}_{W^*}^2
   +R^*(\mu)-R^*(\mu_\star)-\pair{\mu-\mu_\star}{Ty_\star}.
\end{aligned}
\]
Their sum is \eqref{eq:exact_error_identity}.
The first remainder is nonnegative by the subgradient inequality;
the second follows from
$R^*(\mu)\ge\pair{\mu}{Ty_\star}-R(Ty_\star)$ and the displayed
Fenchel equality.
\end{proof}

The identity \eqref{eq:exact_error_identity} expresses the gap as the
sum of the primal and dual objective errors, with their quadratic
contributions displayed explicitly. This is a specialization of the 
generalized Prager--Synge identities to the proximal problem
\citep{BartelsKaltenbach2024,AntilBartelsKaltenbachKhandelwal2025}.
The exact solution is used to interpret the gap; evaluating the gap
requires only the candidate pair. The outer results use
Proposition~\ref{prop:gap} and do not assume existence of $\mu_\star$. 

The factor-two improvement of \eqref{eq:recovered_gap_distance}
over \eqref{eq:gap_distance} is sharp even
for a finite Lipschitz regularizer. Take $W=X=\R$, $T=I$,
$R(x)=|x|$, and $q=t$. Then $y_\star=0$; for $\mu=-1$,
the recovered point is $y=2t$ and
$\mathcal G_{q,t}(y,\mu)=4t=\norm{y-y_\star}^2/t$.

The two-distance bound also gives a ball containing $y_\star$.
The recovered point is $y-t\RW^{-1}r$, so the parallelogram identity
in \eqref{eq:two_distance_bound} gives
\begin{equation}\label{eq:shifted_error_bound}
 \left\|y-\frac t2\RW^{-1}r-y_\star\right\|_W^2
 \le t\mathcal G_{q,t}(y,\mu)-\frac{t^2}{4}\norm{r}_{W^*}^2.
\end{equation}
The right-hand side is nonnegative by \eqref{eq:gap_identity}.
For $r=0$, this is \eqref{eq:recovered_gap_distance}.

These sharper bounds supplement the original certificate below. 
They apply to pairs with finite gaps and require
$Ty\in\dom R$. In particular, for the indicator
$R=\delta_C$ of a nonempty closed convex set $C\subseteq X$,
a finite gap still requires $Ty\in C$.
The shifted center need not itself be feasible; it describes the
error bound and is not automatically an admissible trial point.

For $\epsilon\ge0$, the Fenchel defect also has the interpretation
\begin{equation}\label{eq:epsilon_graph}
 \mu\in\partial_\epsilon R(Ty)
 \quad\Longleftrightarrow\quad
 \varepsilon_R\le\epsilon,
\end{equation}
where $\mu\in\partial_\epsilon R(x)$ means
$R(v)\ge R(x)+\pair{\mu}{v-x}-\epsilon$ for every $v\in X$ and
$x\in\dom R$. Indeed, taking the supremum of
$\pair{\mu}{v}-R(v)$ gives the forward implication, and Fenchel--Young
gives the reverse implication.

\begin{definition}[Relative certificate]\label{def:certificate}
Fix $0<\sigma<1$. For $w\in\dom\varphi$, $t>0$, and $q=w-t\nabla f(w)$, a pair $(s,\mu)$ is certified if its gap is finite and
\begin{equation}\label{eq:certificate}
 \mathcal G_{q,t}(w+s,\mu)\le\frac{\sigma^2}{2t}\norm{s}_W^2.
\end{equation}
If a rigorous upper bound replaces the gap in this test, use that
same bound in \eqref{eq:upper_criticality} and both tests of
Algorithm~\ref{alg:reference}. The results below apply to either choice.
\end{definition}

\begin{theorem}[Distance, descent, and model decrease]\label{thm:certificate}
Let $(s,\mu)$ be certified and set $p=p(w;t)$.
Then
\begin{gather}
 \norm{s-p}_W\le\sigma\norm{s}_W,\qquad
 \frac{\norm{p}_W}{1+\sigma}\le\norm{s}_W
       \le\frac{\norm{p}_W}{1-\sigma},\label{eq:norm_comparison}\\
 \ip{\nabla f(w)}{s}+\varphi(w+s)-\varphi(w)
       \le-\frac{1-\sigma}{t}\norm{s}_W^2.\label{eq:linear_decrease}
 \intertext{If $B:W\to W$ is bounded and self-adjoint and $t\norm{B}\le1-\sigma$,}
 -\left(\ip{\nabla f(w)}{s}+\tfrac12\ip{Bs}{s}
          +\varphi(w+s)-\varphi(w)\right)
       \ge\frac{1-\sigma}{2t}\norm{s}_W^2.\label{eq:model_decrease}
\end{gather}
In particular, $p=0$ forces $s=0$ and zero gap.
\end{theorem}
\begin{proof}
Apply \eqref{eq:gap_distance} to $y=w+s$ and $y_\star=w+p$.
The triangle inequality gives both norm comparisons.
The $\varepsilon_R$-subgradient inequality at $T(w+s)$, evaluated at $Tw$,
gives
\[
 \ip{\nabla f(w)}{s}+\varphi(w+s)-\varphi(w)
 \le-\frac1t\norm{s}_W^2+\pair{r}{s}+\varepsilon_R.
\]
Weighted Young's inequality and $0<\sigma<1$ imply
\[
 \begin{aligned}
 \pair{r}{s}+\varepsilon_R
 &\le\frac{t}{2\sigma}\norm{r}_{W^*}^2
          +\frac\sigma{2t}\norm{s}_W^2+\varepsilon_R\\
 &\le\frac1\sigma\mathcal G_{q,t}(w+s,\mu)
          +\frac\sigma{2t}\norm{s}_W^2
 \le\frac\sigma t\norm{s}_W^2.
 \end{aligned}
\]
This proves \eqref{eq:linear_decrease}. Subtracting a quadratic of size at
most $\norm{B}\norm{s}_W^2/2$ proves \eqref{eq:model_decrease}.
If $p=0$, \eqref{eq:norm_comparison} gives $s=0$, and
\eqref{eq:certificate} forces zero gap.
\end{proof}

For a certified recovered pair, the preceding bounds can be sharpened
without changing the test. Equations~\eqref{eq:certificate} and
\eqref{eq:recovered_gap_distance} give
$\norm{s-p}_W\le(\sigma/\sqrt2)\norm{s}_W$.
Moreover, $r=0$ and $\varepsilon_R=\mathcal G_{q,t}(w+s,\mu)$
in the proof of Theorem~\ref{thm:certificate} give
\begin{equation}\label{eq:recovered_linear_decrease}
 \ip{\nabla f(w)}{s}+\varphi(w+s)-\varphi(w)
 \le-\frac{1-\sigma^2/2}{t}\norm{s}_W^2.
\end{equation}
The outer analysis retains the constants in
Theorem~\ref{thm:certificate}, which apply also when $r\ne0$.

\begin{remark}[The classical relative-error inequality]
For the exact gap in \eqref{eq:gap_identity}, condition \eqref{eq:certificate} is equivalently
\[
 \norm{s+t\nabla f(w)+t\RW^{-1}T^*\mu}_W^2
       +2t\varepsilon_R\le\sigma^2\norm{s}_W^2.
\]
Testing a rigorous upper bound for the gap still implies this inequality,
but need not be equivalent to it.
Furthermore, $f'(w)+T^*\mu$ is an $\varepsilon_R$-subgradient of
$v\mapsto\ip{\nabla f(w)}{v}+\varphi(v)$ at $w+s$, by substitution
in the subgradient inequality. This is the classical relative-error form
associated with enlarged subdifferentials \citep{SolodovSvaiter1999}.
The outer method below uses the step satisfying this condition
as a trust-region comparison step.
\end{remark}

\subsection{A computable stationarity test}
At a stationary outer point, the relative condition \eqref{eq:certificate} requires a zero
step and zero gap. An asymptotically convergent
inner method need not produce those exactly in finitely many iterations.
For a positive stopping tolerance, we therefore first test an absolute bound.
For general extended-valued regularizers, a zero-gap pair
need not exist. For example, take $W=X=\R$, $T=0$, $f=0$, and
$R(x)=-\sqrt{x}$ for $x\ge0$, with $R(x)=+\infty$ otherwise.
Then $F=f+R\circ T\equiv0$, so every outer point is stationary,
whereas $R^*(\mu)=-1/(4\mu)>0$ for $\mu<0$ and
$R^*(\mu)=+\infty$ otherwise. At $q=w$, the dual supremum
is zero but is not attained. The pairs
$(y^j,\mu^j)=(w,-j)$, $j\ge1$, have gap $1/(4j)\to0$,
although no finite pair has zero gap. Thus the relative
condition cannot be satisfied. Along this sequence, however,
the absolute bound defined below satisfies
$U_\tau(y^j,\mu^j;w)=(2\tau j)^{-1/2}$, so every positive
stopping tolerance is met after finitely many iterations.
At $q=w-\tau \nabla f(w)$, define
\begin{equation}\label{eq:upper_criticality}
 U_\tau(y,\mu;w)\defeq 
 \frac{\norm{y-w}_W}{\tau}
 +\sqrt{\frac{2\mathcal G_{q,\tau}(y,\mu)}{\tau}}.
\end{equation}
Equation~\eqref{eq:gap_distance} immediately gives $\chi(w)\le U_\tau$.
For a recovered pair, \eqref{eq:recovered_gap_distance} also yields
$\chi(w)\le\norm{y-w}_W/\tau+
\sqrt{\mathcal G_{q,\tau}(y,\mu)/\tau}$.
We keep \eqref{eq:upper_criticality} in the algorithm so that the same
test applies to every pair with a finite gap. 

\begin{assumption}[Inner approximation property]\label{ass:oracle}
For every requested $w\in\dom\varphi$ and $t>0$, set
$q=w-t\nabla f(w)$. The inner procedure produces pairs
$(y^j,\mu^j)$ with finite gaps and
$\mathcal G_{q,t}(y^j,\mu^j)\to0$.
The gap or a rigorous upper bound tending to zero is available for testing.
\end{assumption}
This is an assumption on the inner procedure and its computable
certificates. It requires neither dual attainment nor computation of an exact
proximal point, and imposes no prescribed convergence rate. For extended-valued
$R$, maintaining $Ty^j\in\dom R$ is part of the requirement.
For each request to the inner procedure, $w$, $t$, and $q$
remain fixed. The pairs $(y^j,\mu^j)$ are primal and dual
candidates for the problems in \eqref{eq:primal_dual};
the unique primal minimizer is $w+p(w;t)$.
Section~\ref{sec:inner} constructs such pairs with gaps tending
to zero when $X$ is Hilbert and
Assumption~\ref{ass:lipschitz_R} holds. The following algorithm
tests these pairs at the reference parameter $t=\tau$.

\Needspace{9\baselineskip}
\begin{algorithm}[Stationarity test and relative approximation]\label{alg:reference}
Given $w\in\dom\varphi$, a positive tolerance $\varepsilon$, and $q=w-\tau \nabla f(w)$, generate
the pairs in Assumption~\ref{ass:oracle}. At each pair:
\begin{enumerate}
\item[(i)] If $U_\tau(y^j,\mu^j;w)\le\varepsilon$, return a stationarity
certificate $\chi(w)\le\varepsilon$.
\item[(ii)] Otherwise, if \eqref{eq:certificate} holds with $s=y^j-w$ and $t=\tau$,
return this relative certificate and continue the outer iteration.
\end{enumerate}
\end{algorithm}

\begin{proposition}[Finite termination of the stationarity test]\label{prop:screening}
Under Assumption~\ref{ass:oracle}, Algorithm~\ref{alg:reference} stops after
finitely many inner iterations. Termination by the first test guarantees
$\chi(w)\le\varepsilon$. Termination by the second test guarantees
\begin{equation}\label{eq:continue_criticality}
 \chi(w)>c_\sigma\varepsilon,
 \qquad c_\sigma\defeq \frac{1-\sigma}{1+\sigma}>0.
\end{equation}
At any requested $t>0$ with $p(w;t)\ne0$, the relative test alone succeeds
after finitely many inner iterations.
\end{proposition}
\begin{proof}
Equation~\eqref{eq:gap_distance} and
Assumption~\ref{ass:oracle} give $y^j\to w+p(w;t)$.
If $p(w;t)\ne0$, the right-hand side of the relative test tends to
$\sigma^2\norm{p(w;t)}_W^2/(2t)>0$, whereas the tested gap bound
tends to zero. Thus the relative test succeeds finitely.
If $\chi(w)=0$, the reference request instead has $y^j\to w$ and
$U_\tau\to0$, so its positive absolute tolerance is reached finitely.
These cases prove termination, and \eqref{eq:upper_criticality}
gives the absolute guarantee.

When the second test succeeds, the first has failed, and hence
\[
 \varepsilon<U_\tau\le\frac{1+\sigma}{\tau}\norm{y^j-w}_W,
 \qquad
 \chi(w)\ge\frac{1-\sigma}{\tau}\norm{y^j-w}_W.
\]
Their combination proves \eqref{eq:continue_criticality}.
\end{proof}

The second alternative need not imply $\chi(w)>\varepsilon$; the lower
bound $c_\sigma\varepsilon$ suffices for the iteration estimate. The
returned relative pair justifies continuation, while the Cauchy search
below selects its own proximal parameter.

\section{Cauchy steps and trust-region convergence}
\label{sec:cauchy}
At $w^k\in\dom\varphi$, choose a bounded self-adjoint operator
$B_k:W\to W$ and define
\begin{equation}\label{eq:model}
 m_k(s)\defeq f(w^k)+\ip{\nabla f(w^k)}{s}
       +\tfrac12\ip{B_ks}{s}+\varphi(w^k+s),
 \qquad \pred_k(s)\defeq m_k(0)-m_k(s).
\end{equation}
The operator $B_k$ may be indefinite. We first construct a step with
guaranteed model decrease in the ball $\norm{s}_W\le\Delta_k$, where
$\Delta_k>0$. This construction does not require the full trust-region
subproblem to have a minimizer.

\begin{algorithm}[Geometric Cauchy search]\label{alg:cauchy}
Fix \(0<\zeta<1\). At a point with \(\chi(w^k)>0\), set
\(\bar t_k=\min\{\tau,(1-\sigma)/(1+\norm{B_k})\}\).
For \(j=0,1,\ldots\), let \(t_{k,j}=\zeta^j\bar t_k\) and
set \(q=w^k-t_{k,j}\nabla f(w^k)\).
At each parameter, use the procedure in Assumption~\ref{ass:oracle}
until a pair $(s_{k,j},\mu_{k,j})$ satisfies \eqref{eq:certificate}
at $w=w^k$ and $t=t_{k,j}$.
Return the first pair for which
\(\norm{s_{k,j}}_W\le\Delta_k\), and denote its step
and parameter by \(s_k^{\rm C}\) and \(t_k\).
\end{algorithm}
Throughout this search, \(w^k\), \(B_k\), and \(\Delta_k\)
remain fixed. The index \(j\) counts reductions of the
proximal parameter; each parameter choice may require
several inner iterations.

\begin{theorem}[Finite Cauchy search and decrease]\label{thm:cauchy}
Suppose Assumptions~\ref{ass:basic} and \ref{ass:oracle} hold.
Algorithm~\ref{alg:cauchy} terminates after finitely many inner iterations
and parameter reductions. Its returned point $w^k+s_k^{\rm C}$ belongs to $\dom\varphi$,
and, with $\chi_k\defeq\chi(w^k)$,
\begin{equation}\label{eq:cauchy_decrease}
 \pred_k(s_k^{\rm C})\ge
 \kappa_{\rm C}\chi_k
 \min\left\{\frac{\chi_k}{1+\norm{B_k}},\Delta_k\right\},
\end{equation}
where
\begin{equation}\label{eq:cauchy_constant}
 \kappa_{\rm C}\defeq
 \frac{(1-\sigma)\min\{\tau,\zeta(1-\sigma)\}}
      {2(1+\sigma)^2}>0.
\end{equation}
\end{theorem}
\begin{proof}
Write $p_k(t)=p(w^k;t)$. Since $\chi_k>0$, Lemma~\ref{lem:prox}
gives $p_k(t)\ne0$ for every $t>0$, so each relative test terminates
finitely by Proposition~\ref{prop:screening}. Moreover,
$\norm{s_{k,j}}_W\le\norm{p_k(t_{k,j})}_W/(1-\sigma)\to0$.
Thus the radius test also succeeds after finitely many reductions.
The returned point $w^k+s_k^{\rm C}$ is in $\dom\varphi$ because 
its gap is finite.

Set $s=s_k^{\rm C}$ and $t=t_k$. The inequalities
$t\norm{B_k}\le1-\sigma$ and $t\le\tau$, together with
Theorem~\ref{thm:certificate} and Lemma~\ref{lem:prox}, give
\begin{equation}\label{eq:cauchy_two_basic}
 \pred_k(s)\ge\frac{1-\sigma}{2t}\norm{s}_W^2,
 \qquad
 \frac{\norm{s}_W}{t}
 \ge\frac{\norm{p_k(t)}_W}{(1+\sigma)t}
 \ge\frac{\chi_k}{1+\sigma}.
\end{equation}
If $t=\bar t_k$, then
$t\ge\min\{\tau,1-\sigma\}/(1+\norm{B_k})$, and hence
\[
 \pred_k(s)\ge
 \frac{(1-\sigma)\min\{\tau,1-\sigma\}}{2(1+\sigma)^2}
 \frac{\chi_k^2}{1+\norm{B_k}}.
\]
Otherwise the preceding certified step, at $t/\zeta$, has norm
greater than $\Delta_k$. The norm comparisons and
\eqref{eq:path_norms} therefore yield
\[
 \norm{s}_W\ge\frac{\norm{p_k(t)}_W}{1+\sigma}
 \ge\frac{\zeta\norm{p_k(t/\zeta)}_W}{1+\sigma}
 >\frac{\zeta(1-\sigma)}{1+\sigma}\Delta_k.
\]
Using this for one factor of $\norm{s}_W$ and the second estimate
in \eqref{eq:cauchy_two_basic} for the other gives
\[
 \pred_k(s)\ge
 \frac{\zeta(1-\sigma)^2}{2(1+\sigma)^2}\chi_k\Delta_k.
\]
Taking the smaller of the two constants proves
\eqref{eq:cauchy_decrease}--\eqref{eq:cauchy_constant}.
\end{proof}

A known upper bound $b_k\ge\norm{B_k}$ may replace $\norm{B_k}$ in the
Cauchy search and in \eqref{eq:cauchy_decrease}. The proof uses only an
upper bound on the quadratic term. In this case, the uniform model bound
in the outer convergence results below is imposed on $b_k$.

One may replace $s_k^{\rm C}$ by an alternative trial step $d_k$
satisfying, for a fixed $0<\nu\le1$,
\begin{equation}\label{eq:improved_step}
 \norm{d_k}_W\le\Delta_k,\qquad
 w^k+d_k\in\dom\varphi,\qquad
 \pred_k(d_k)\ge\nu\pred_k(s_k^{\rm C}).
\end{equation}
The Cauchy step is always admissible; Newton or quasi-Newton steps
may replace it whenever \eqref{eq:improved_step} holds.
Write $\kappa\defeq\nu\kappa_{\rm C}$ throughout this section.

\Needspace{8\baselineskip}
\begin{algorithm}[Trust-region method with relative gap tests]\label{alg:outer}
Choose $w^0\in\dom\varphi$, $0<\Delta_0\le\Delta_{\max}<\infty$, a
positive tolerance $\varepsilon$, and parameters
$0<\eta_1<\eta_2<1$, $0<\gamma_c<1<\gamma_e$.
For $k=0,1,\ldots$:
\begin{enumerate}
\item Apply Algorithm~\ref{alg:reference} at $w=w^k$.
If its stationarity test gives $U_\tau\le\varepsilon$,
terminate and return $w^k$. Otherwise, it returns a relative
certificate, and Proposition~\ref{prop:screening} gives
$\chi_k>\frac{1-\sigma}{1+\sigma}\varepsilon>0$.
\item Choose $B_k$, compute $s_k^{\rm C}$ by Algorithm~\ref{alg:cauchy}, and
choose $d_k$ satisfying \eqref{eq:improved_step}.
The Cauchy decrease bound ensures $\pred_k(d_k)>0$.
\item Compute
$\rho_k=[F(w^k)-F(w^k+d_k)]/\pred_k(d_k)$.
\item If $\rho_k<\eta_1$, set
$w^{k+1}=w^k$ and $\Delta_{k+1}=\gamma_c\Delta_k$.
If $\eta_1\le\rho_k<\eta_2$, set
$w^{k+1}=w^k+d_k$ and $\Delta_{k+1}=\Delta_k$.
If $\rho_k\ge\eta_2$, set $w^{k+1}=w^k+d_k$ and
$\Delta_{k+1}=\min\{\gamma_e\Delta_k,\Delta_{\max}\}$.
\end{enumerate}
\end{algorithm}

An iteration is \emph{successful} if $\rho_k\ge\eta_1$ and
\emph{very successful} if $\rho_k\ge\eta_2$.
Proposition~\ref{prop:screening} ensures $\chi_k>0$ whenever a trial is
formed. Consequently, \eqref{eq:cauchy_decrease} and
\eqref{eq:improved_step} give $\pred_k(d_k)>0$, so the ratio is defined.
Every iterate and trial point has finite objective, and every
successful iteration decreases $F$.
The exact step and $\chi$ appear in the analysis, but their exact
values are not required by the algorithm. The stopping test uses the gap
and $U_\tau$; acceptance uses actual and predicted reduction. Neither
$F_{\rm low}$ nor the Lipschitz constant $L_f$ is an algorithmic input.

\subsection{Model agreement and lower bounds for the radius}
Throughout the remaining analysis, assume
$\norm{B_k}\le b_{\max}<\infty$. The usual trust-region comparison between
actual and predicted reduction applies because the nonsmooth term is
included exactly in the model; see also \citet{ConnGouldToint2000}.

\begin{lemma}[Model agreement]\label{lem:agreement}
For every trial step,
\begin{equation}\label{eq:agreement}
 \bigl|F(w^k)-F(w^k+d_k)-\pred_k(d_k)\bigr|
 \le\frac{L_f+b_{\max}}2\norm{d_k}_W^2.
\end{equation}
In particular, let
\begin{equation}\label{eq:radius_threshold}
 c\defeq\min\left\{\frac1{1+b_{\max}},
             \frac{2\kappa(1-\eta_2)}{1+L_f+b_{\max}}\right\}>0.
\end{equation}
Every iteration satisfying $\Delta_k\le c\chi_k$ is very successful.
\end{lemma}
\begin{proof}
The nonsmooth terms cancel in actual minus predicted reduction.
Integration of $\nabla f$ along the trial segment gives
$|f(w^k+d_k)-f(w^k)-\ip{\nabla f(w^k)}{d_k}|
\le L_f\norm{d_k}_W^2/2$.
Adding $|\ip{B_kd_k}{d_k}|/2\le b_{\max}\norm{d_k}_W^2/2$
proves \eqref{eq:agreement}. If $\Delta_k\le c\chi_k$, then
\eqref{eq:cauchy_decrease} and \eqref{eq:improved_step} give
$\pred_k(d_k)\ge\kappa\chi_k\Delta_k$. Consequently,
\[
 |1-\rho_k|\le
 \frac{(L_f+b_{\max})\Delta_k}{2\kappa\chi_k}
 \le(1-\eta_2)\frac{L_f+b_{\max}}{1+L_f+b_{\max}}
 \le1-\eta_2,
\]
and the trial is very successful.
\end{proof}

\begin{lemma}[Radius bound away from stationarity]\label{lem:radius_floor}
If $\chi_k\ge\delta>0$ for $k=K,\ldots,L-1$, then
\begin{equation}\label{eq:radius_floor}
 \Delta_k\ge\min\{\Delta_K,\gamma_c c\delta\}
 \qquad (K\le k\le L).
\end{equation}
\end{lemma}
\begin{proof}
Successful trials do not decrease the radius. At an unsuccessful trial
with $K\le k<L$, Lemma~\ref{lem:agreement} gives
$\Delta_k>c\chi_k\ge c\delta$, so
$\Delta_{k+1}=\gamma_c\Delta_k>\gamma_c c\delta$.
Starting from $\Delta_K$, induction proves the bound through index $L$;
no condition on $\chi_L$ is needed.
\end{proof}

\subsection{Convergence to stationarity}
We first analyze the step construction and radius updates without
imposing the positive-tolerance stopping test. Specifically, we show
that any infinite sequence of nonstationary iterates satisfying these
rules has $\chi(w^k)\to0$. This does not require an inner procedure
to certify exact stationarity in finitely many iterations.

\begin{theorem}[Global convergence]\label{thm:global}
Suppose Assumptions~\ref{ass:basic} and \ref{ass:oracle} hold and
$\norm{B_k}\le b_{\max}$. Let an infinite sequence satisfy 
$w^0\in\dom\varphi$, $0<\Delta_0\le\Delta_{\max}$, $\chi(w^k)>0$,
Algorithm~\ref{alg:cauchy}, \eqref{eq:improved_step}, and the ratio and
update rules in steps 3--4 of Algorithm~\ref{alg:outer}. Then
\begin{equation}\label{eq:global_chi}
 \chi(w^k)\longrightarrow0.
\end{equation}
Every strong accumulation point is stationary. In finite dimensions,
every accumulation point is therefore stationary, and boundedness of
the sequence guarantees the existence of at least one such point.
\end{theorem}
\begin{proof}
The objective values decrease monotonically to a finite limit.
There are infinitely many successful trials: otherwise $w^k$ and its
positive criticality would eventually be fixed while repeated rejection
forces $\Delta_k\to0$, contradicting Lemma~\ref{lem:agreement}.
If $\chi_k\ge\delta>0$ for every $k\ge K$.
Lemma~\ref{lem:radius_floor} gives
$\Delta_k\ge\min\{\Delta_K,\gamma_c c\delta\}>0$ for all $k\ge K$.
At every successful iteration with $k \ge K$, the acceptance condition ($\rho_k \ge \eta_1$),
\eqref{eq:cauchy_decrease}, and \eqref{eq:improved_step} give
\[
 F(w^k)-F(w^{k+1})\ge\eta_1\kappa\delta
 \min\left\{\frac{\delta}{1+b_{\max}},\Delta_K,
                    \gamma_c c\delta\right\}>0.
\]
Infinitely many such decreases contradict the lower bound on $F$.
Thus $\liminf_k\chi_k=0$.

Fix $\delta>0$. At a successful trial with $\chi_k\ge\delta$,
\[
 F(w^k)-F(w^{k+1})\ge\eta_1\kappa\delta
 \min\left\{\frac{\delta}{1+b_{\max}},\Delta_k\right\}.
\]
Since successive objective decreases tend to zero, all sufficiently
late such trials have $\Delta_k\le\delta/(1+b_{\max})$. Hence
using $w^{k+1}-w^k=d_k$ and $\norm{d_k}_W\le\Delta_k$
\begin{equation}\label{eq:movement_descent}
 F(w^k)-F(w^{k+1})\ge\eta_1\kappa\delta\Delta_k
 \ge\eta_1\kappa\delta\norm{w^{k+1}-w^k}_W.
\end{equation}
The final inequality between objective decrease and movement also holds
at rejected trials, where both quantities vanish.

If $\limsup_k\chi_k>0$, choose $\delta>0$ with
$2\delta<\limsup_k\chi_k$. Because $\liminf_k\chi_k=0$, there are
arbitrarily late intervals $[i,j]$ with $\chi_i\ge2\delta$,
$\chi_j<\delta$, and $\chi_k\ge\delta$ for $i\le k<j$.
Lipschitz continuity of $\chi$ gives
\[
 \sum_{k=i}^{j-1}\norm{w^{k+1}-w^k}_W
 \ge\frac{|\chi_i-\chi_j|}{L_\chi}>\frac{\delta}{L_\chi}.
\]
Summing \eqref{eq:movement_descent} on a sufficiently late interval gives
$F(w^i)-F(w^j)>\eta_1\kappa\delta^2/L_\chi$.
This contradicts convergence of $F(w^k)$, proving
\eqref{eq:global_chi}. 
Continuity of $\chi$ and Lemma~\ref{lem:chi} make every strong
accumulation point stationary, including its membership in
$\dom\varphi$. In finite dimensions, every bounded sequence has a
convergent subsequence.
\end{proof}

The theorem concerns strong accumulation points. Weak subsequential
compactness alone does not justify passage to the limit in the generally
nonlinear proximal stationarity map. Additional assumptions would be
needed to deduce stationarity of arbitrary weak accumulation points.

\subsection{Finite termination and number of outer iterations}

An outer trial is one execution of steps 2--4 of
Algorithm~\ref{alg:outer}; both accepted and rejected trials are counted.
Any optional search for an alternative trial step is terminated after
finitely many iterations, using $d_k=s_k^{\rm C}$ if no admissible
alternative has been found.

\begin{theorem}[Finite termination for a positive tolerance]\label{thm:complexity}
Suppose Assumptions~\ref{ass:basic} and \ref{ass:oracle} hold.
Fix the initial point, radii, algorithmic parameters, and uniform bound
$\norm{B_k}\le b_{\max}$ independently of $\varepsilon$.
For every $0<\varepsilon\le1$, Algorithm~\ref{alg:outer} terminates
finitely and returns a point with $\chi(w)\le\varepsilon$.
The number of outer trials is $O(\varepsilon^{-2})$ as
$\varepsilon\downarrow0$. This estimate does not count inner proximal
iterations or applications of $T$ and $T^*$.
\end{theorem}
\begin{proof}
Every reference call terminates finitely by
Proposition~\ref{prop:screening}. If it returns a relative certificate,
then $\chi_k>c_\sigma\varepsilon>0$, and the Cauchy search is finite
by Theorem~\ref{thm:cauchy}.

Consider any initial segment of $K$ completed trials. The continuation
bound and Lemma~\ref{lem:radius_floor} give
\begin{equation}\label{eq:complexity_radius_bound}
 \Delta_k\ge\underline{\Delta}_\varepsilon
 \defeq \min\{\Delta_0,\gamma_c c c_\sigma\varepsilon\}>0,
 \qquad 0\le k\le K.
\end{equation}
For $0<\varepsilon\le1$,
$\underline{\Delta}_\varepsilon
\ge\varepsilon\min\{\Delta_0,\gamma_c c c_\sigma\}$.
Thus each successful trial decreases $F$ by at least
$C_s\varepsilon^2$, where
\begin{equation}\label{eq:complexity_constants}
 C_s\defeq \eta_1\kappa c_\sigma
 \min\left\{\frac{c_\sigma}{1+b_{\max}},\Delta_0,
                         \gamma_c c c_\sigma\right\}>0.
\end{equation}
This constant is independent of both $\varepsilon$ and $K$.

Let $N_s$ and $N_u$ denote the numbers of successful and unsuccessful
trials in this segment, so that $K=N_s+N_u$.
Rejected trials leave the objective unchanged, while each successful
trial decreases it by at least $C_s\varepsilon^2$. Summing gives
\[
 N_sC_s\varepsilon^2
 \le
 \sum_{k=0}^{K-1}\bigl[F(w^k)-F(w^{k+1})\bigr]
 =
 F(w^0)-F(w^K)
 \le F(w^0)-F_{\rm low}.
\]

To bound the unsuccessful trials, use the radius updates.
Each unsuccessful trial gives
$\Delta_{k+1}=\gamma_c\Delta_k$.
Each successful trial gives
$\Delta_{k+1}\le\gamma_e\Delta_k$, whether the radius stays unchanged
or increases subject to the upper bound $\Delta_{\max}$.
Multiplying these inequalities over the $K$ trials and using
\eqref{eq:complexity_radius_bound}, we obtain
\[
 \underline{\Delta}_\varepsilon
 \le\Delta_K
 \le\Delta_0\gamma_c^{N_u}\gamma_e^{N_s}.
\]
Taking logarithms and using
$\log\gamma_c=-|\log\gamma_c|$ yields
\[
 N_u|\log\gamma_c|
 \le
 \log(\Delta_0/\underline{\Delta}_\varepsilon)
 +N_s\log\gamma_e.
\]
Together with the objective-decrease estimate, this proves
\begin{equation}\label{eq:complexity}
 \begin{aligned}
 N_s&\le
 \frac{F(w^0)-F_{\rm low}}{C_s\varepsilon^2},\\
 N_u&\le
 \frac{\log(\Delta_0/\underline{\Delta}_\varepsilon)
                  +N_s\log\gamma_e}
      {|\log\gamma_c|}.
 \end{aligned}
\end{equation}
Consequently,
\begin{equation}\label{eq:total_outer_trials}
 \begin{aligned}
 K=N_s+N_u\le{}&
 \left(1+\frac{\log\gamma_e}{|\log\gamma_c|}\right)
 \frac{F(w^0)-F_{\rm low}}{C_s\varepsilon^2}
 +\frac{\log(\Delta_0/\underline{\Delta}_\varepsilon)}
       {|\log\gamma_c|}.
 \end{aligned}
\end{equation}
For fixed $\varepsilon>0$, this bound is finite and independent of
$K$, so an infinite sequence of completed trials is impossible.
Reference calls, Cauchy searches, and the optional step search all take
finite work. The algorithm therefore terminates at the absolute test,
which gives $\chi(w)\le U_\tau\le\varepsilon$.
Finally, $\underline{\Delta}_\varepsilon
=\gamma_c c c_\sigma\varepsilon$ for sufficiently small
$\varepsilon$, so the logarithmic term is
$O(1+|\log\varepsilon|)$. Equation~\eqref{eq:total_outer_trials}
proves the stated $O(\varepsilon^{-2})$ estimate.
\end{proof}

The $O(\varepsilon^{-2})$ dependence is the standard first-order
iteration bound for trust-region methods
\citep{ConnGouldToint2000,BaraldiKouri2023}.
The relative gap test gives the
required model decrease and stationarity bounds while allowing a nonzero
Fenchel defect. It requires neither an exact composite proximal map nor
exact minimization of the possibly nonconvex model.

\section{A dual proximal-gradient inner solver}
\label{sec:inner}

We now verify Assumption~\ref{ass:oracle} for globally Lipschitz
regularizers using dual proximal-gradient iteration. Primal recovery
and gap estimates for linear compositions are classical
\citep[Proposition~2.2 and Theorem~6.1]{VillaSalzoBaldassarreVerri2013}.
Transferring an $O(j^{-1})$ dual objective estimate through the primal
distance gives an $O(j^{-1/2})$ gap estimate. We instead bound the gap
by a dual step and obtain $O(j^{-1})$, as in related finite-dimensional
analyses \citep[Lemma~4.4 and Theorem~4.1]{KimFessler2016}.
We give a Hilbert-space proof and an explicit bound for an averaged
primal candidate.

Throughout this section, $X$ is Hilbert, with Riesz map $\mathcal R_X$
and Hilbert adjoint
$T^{\star}\defeq\RW^{-1}T^*\mathcal R_X:X\to W$.
The variable $\mu$ denotes the Hilbert representative of a dual vector,
and $R^*(\mu)$ means $R^*(\mathcal R_X\mu)$.

\begin{assumption}[Regularizer and required evaluations]\label{ass:lipschitz_R}
The function $R:X\to\R$ is convex and globally Lipschitz with constant
$L_R<\infty$. Its values, finite conjugate values, and the proximal map
of $R^*$ in $X$ can be evaluated. A starting point
$\mu^0\in\dom R^*$ is available.
\end{assumption}

The iteration bounds below use a known $L_R$; the gap tests do not.
Norms and finite sums of boundedly weighted norms satisfy the regularity
hypothesis. So does the support function of a nonempty bounded closed
convex set, provided its support function and metric projection can be
evaluated: its conjugate is the indicator of that set. The assumption
excludes nontrivial indicator constraints on $Ty$.

For fixed $q,t$, minimize the negative dual objective (see \eqref{eq:primal_dual})
\begin{equation}\label{eq:dual_min}
 H(\mu)\defeq h(\mu)+R^*(\mu),\qquad
 h(\mu)\defeq\frac t2\norm{T^{\star}\mu}_W^2-(\mu,Tq)_X.
\end{equation}
Choose $\gamma>0$ such that $\gamma t\norm{T}^2\le1$.
A known upper bound for $\norm{T}$ suffices for this choice.
For $j=0,1,\ldots$, form the primal--dual pair using
\begin{equation}\label{eq:dual_iteration}
 \begin{aligned}
 y^j&=q-tT^{\star}\mu^j,\\
 \mu^{j+1}&=\prox_{\gamma R^*}
       \bigl(\mu^j+\gamma Ty^j\bigr).
 \end{aligned}
\end{equation}
Test $(y^j,\mu^j)$ before performing the dual update. Recovery makes
$r^j=t^{-1}\RW(y^j-q)+T^*\mathcal R_X\mu^j=0$, so
\[
 \mathcal G_{q,t}(y^j,\mu^j)
 =R(Ty^j)+R^*(\mu^j)-(\mu^j,Ty^j)_X.
\]
Every gap is finite because $R$ is finite and the proximal update
preserves $\dom R^*$. It vanishes precisely when
$\mu^j\in\partial R(Ty^j)$, which need not occur at a finite index.

\begin{lemma}[Dual attainment and primal recovery]\label{lem:dual_attainment}
Under Assumption~\ref{ass:lipschitz_R}, $H$ has a minimizer $\mu_\star$,
$\norm{\mu_\star}_X\le L_R$, and
$y_\star=q-tT^{\star}\mu_\star$ is the unique minimizer of
$\Psi_{q,t}$. Moreover, for every $\mu\in\dom R^*$,
\begin{equation}\label{eq:dual_primal_distance}
 H(\mu)-H(\mu_\star)
 \ge\frac t2\norm{T^{{\star}}(\mu-\mu_\star)}_W^2.
\end{equation}
\end{lemma}
\begin{proof}
If $\norm{\mu}_X>L_R$, testing the conjugate at
$x=s\mu/\norm{\mu}_X$ gives
$R^*(\mu)\ge s(\norm{\mu}_X-L_R)-R(0)\to+\infty$ as
$s\to\infty$. Thus $\dom R^*\subseteq\{\norm{\mu}_X\le L_R\}$.
Also $R^*(\mu)\ge-R(0)$, so
$H(\mu)\ge-R(0)-L_R\norm{Tq}_X$ on its domain. Since $H(\mu^0)$
is finite, a minimizing sequence with finite values exists and is
bounded. Hilbert reflexivity gives a weakly convergent subsequence.
Both $R^*$, a supremum of weakly continuous affine functions, and $h$
are weakly lower semicontinuous. Hence its limit $\mu_\star$ minimizes
$H$ and has norm at most $L_R$.

The convex sum rule gives
$-\nabla h(\mu_\star)=Ty_\star\in\partial R^*(\mu_\star)$, where
$y_\star=q-tT^{{\star}}\mu_\star$. Equivalently,
$\mu_\star\in\partial R(Ty_\star)$. The recovered residual and
Fenchel defect both vanish, so the gap is zero and
Proposition~\ref{prop:gap} identifies $y_\star$ as the unique primal
minimizer. Finally, the quadratic expansion of $h$ at $\mu_\star$
and the subgradient inequality for $R^*$ cancel their linear terms
and give \eqref{eq:dual_primal_distance}.
\end{proof}

\begin{theorem}[Vanishing gaps and finite inner termination]\label{thm:inner}
Under Assumption~\ref{ass:lipschitz_R}, iteration
\eqref{eq:dual_iteration}, with $\gamma t\norm{T}^2\le1$,
produces finite gaps tending to zero. For every $j\ge1$,
\begin{align}
 0\le H(\mu^j)-H(\mu_\star)
 &\le\frac{\norm{\mu^0-\mu_\star}_X^2}{2\gamma j},
 \label{eq:dual_rate}\\
 \norm{y^j-y_\star}_W
 &\le\sqrt{\frac{t}{\gamma j}}\norm{\mu^0-\mu_\star}_X,
 \label{eq:primal_rate}\\
 0\le\mathcal G_{q,t}(y^j,\mu^j)
 &\le\frac{2L_R}{\gamma}\norm{\mu^j-\mu^{j-1}}_X.
 \label{eq:gap_dual_step}
\end{align}
In particular, for every $j\ge2$,
\begin{equation}\label{eq:gap_rate}
 \mathcal G_{q,t}(y^j,\mu^j)
 \le\frac{4L_R\norm{\mu^0-\mu_\star}_X}{\gamma(j-1)}.
\end{equation}
Thus Assumption~\ref{ass:oracle} holds. Algorithm~\ref{alg:reference}
and every relative test at a nonstationary point terminate after
finitely many inner iterations.
\end{theorem}
\begin{proof}
Write $L\defeq t\norm{T}^2$ and $D\defeq\norm{\mu^0-\mu_\star}_X$.
For an update $\mu^+=\prox_{\gamma R^*}(\mu-\gamma\nabla h(\mu))$,
proximal optimality gives
$(\mu-\mu^+)/\gamma-\nabla h(\mu)\in\partial R^*(\mu^+)$.
Its subgradient inequality, convexity of $h$ at $\mu$, and the
$L$-smooth upper bound at $\mu^+$ yield, for $v\in\dom R^*$,
\begin{equation}\label{eq:dual_pg_inequality}
 \begin{aligned}
 H(\mu^+)-H(v)
 &\le\frac1\gamma(\mu-\mu^+,\mu^+-v)_X
       +\frac L2\norm{\mu^+-\mu}_X^2\\
 &=\frac{\norm{\mu-v}_X^2-\norm{\mu^+-v}_X^2}{2\gamma}
   -\frac{1-\gamma L}{2\gamma}\norm{\mu^+-\mu}_X^2.
 \end{aligned}
\end{equation}
Taking $v=\mu$ proves
\begin{equation}\label{eq:dual_step_descent}
 H(\mu)-H(\mu^+)
 \ge\left(\frac1\gamma-\frac L2\right)\norm{\mu^+-\mu}_X^2
 \ge\frac{\norm{\mu^+-\mu}_X^2}{2\gamma}.
\end{equation}
Taking $v=\mu_\star$ and summing the first $j$ updates gives
$\sum_{i=1}^j[H(\mu^i)-H(\mu_\star)]\le D^2/(2\gamma)$.
These errors are nonnegative and nonincreasing, proving
\eqref{eq:dual_rate}. Equation~\eqref{eq:dual_primal_distance} and
$y^j-y_\star=-tT^{{\star}}(\mu^j-\mu_\star)$ give
\eqref{eq:primal_rate}.

We next estimate the full gap directly from the preceding dual update.
For a fixed $j\ge1$, set 
$x\defeq Ty^{j-1}+(\mu^{j-1}-\mu^j)/\gamma$.
Proximal optimality gives $x\in\partial R^*(\mu^j)$,
equivalently $\mu^j\in\partial R(x)$, and hence
$R^*(\mu^j)=(\mu^j,x)_X-R(x)$.
The residual at the recovered pair is zero. Therefore
\[
\begin{aligned}
 \mathcal G_{q,t}(y^j,\mu^j)
 &=R(Ty^j)-R(x)-(\mu^j,Ty^j-x)_X\\
 &\le(L_R+\norm{\mu^j}_X)\norm{Ty^j-x}_X
 \le2L_R\norm{Ty^j-x}_X.
\end{aligned}
\]
The recovery formula also gives
$x-Ty^j=(\gamma^{-1}I-tTT^{{\star}})(\mu^{j-1}-\mu^j)$.
Since $tTT^{{\star}}$ is positive semidefinite and
$0\le\gamma tTT^{{\star}}\le I$ in the order of self-adjoint operators,
$\norm{\gamma^{-1}I-tTT^{{\star}}}\le\gamma^{-1} $ by the
spectral bound for positive self-adjoint operators.
This proves \eqref{eq:gap_dual_step}.

To obtain a bound in terms of $j$, observe that the affine map
$\mu\mapsto(I-\gamma tTT^{{\star}})\mu+\gamma Tq$ is
nonexpansive, as is $\prox_{\gamma R^*}$.
Their composition is the dual update map. Applying its nonexpansivity
to consecutive dual iterates $\mu^{i+1} = \prox_{\gamma R^*}\left((I-\gamma t TT^{{\star}})\mu^i + \gamma Tq \right)$ shows that
$\norm{\mu^{i+1}-\mu^i}_X$ is nonincreasing in $i$.
For $j\ge2$, let $m=\lfloor j/2\rfloor\ge1$.
Summing \eqref{eq:dual_step_descent} from $i=m$ to $j-1$ and then
using \eqref{eq:dual_rate} at index $m$ gives
\[
 \frac{j-m}{2\gamma}\norm{\mu^j-\mu^{j-1}}_X^2
 \le H(\mu^m)-H(\mu^j)
 \le\frac{\norm{\mu^0-\mu_\star}_X^2}{2\gamma m}.
\]
Because $m(j-m)\ge(j-1)^2/4$, it follows that
$\norm{\mu^j-\mu^{j-1}}_X
 \le2\norm{\mu^0-\mu_\star}_X/(j-1)$.
Together with \eqref{eq:gap_dual_step}, this proves \eqref{eq:gap_rate}.

Each $R^*(\mu^j)$ is finite by the initial choice and proximal
optimality, and each $R(Ty^j)$ is finite by
Assumption~\ref{ass:lipschitz_R}. Thus all gaps are finite, and
\eqref{eq:gap_rate} shows that they tend to zero.
Proposition~\ref{prop:screening} now gives finite termination of the
stated tests.
\end{proof}

The proof bounds the Fenchel defect directly using the bounded conjugate
domain. It does not require strong convergence of the dual iterates;
the primal points converge strongly by \eqref{eq:primal_rate}.

\subsection{An optional averaged primal candidate}

The average of the preceding recovered points, paired with the current
dual iterate, also gives a vanishing full gap. Its residual $r$ need
not vanish.

\begin{theorem}[Gap bound for an averaged primal candidate]
\label{thm:averaged_gap}
Under the hypotheses of Theorem~\ref{thm:inner}, define, for $j\ge1$,
\begin{equation}\label{eq:primal_average}
 \overline y^{\,j}\defeq\frac1j\sum_{i=0}^{j-1}y^i.
\end{equation}
Then
\begin{equation}\label{eq:averaged_gap_rate}
 0\le\mathcal G_{q,t}(\overline y^{\,j},\mu^j)
 \le\frac{(\norm{\mu^0}_X+L_R)^2}{2\gamma j}
 \le\frac{2L_R^2}{\gamma j}.
\end{equation}
In particular, the pairs $(\overline y^{\,j},\mu^j)$ also satisfy
Assumption~\ref{ass:oracle}.
\end{theorem}
\begin{proof}
For $v\in\dom R^*$, proximal optimality gives
\[
 \begin{aligned}
 R^*(\mu^{i+1})-R^*(v)-(Ty^i,\mu^{i+1}-v)_X
 &\le\frac1\gamma(\mu^i-\mu^{i+1},\mu^{i+1}-v)_X\\
 &=\frac{\norm{\mu^i-v}_X^2-\norm{\mu^{i+1}-v}_X^2
       -\norm{\mu^{i+1}-\mu^i}_X^2}{2\gamma}.
 \end{aligned}
\]
Using $y^i=q-tT^{{\star}}\mu^i$ to expand the quadratics, we obtain
\begin{equation}\label{eq:averaged_one_step}
 \begin{aligned}
 &\frac1{2t}\norm{y^i-q}_W^2+(Ty^i,v)_X-R^*(v)+H(\mu^{i+1})\\
 &=R^*(\mu^{i+1})-R^*(v)-(Ty^i,\mu^{i+1}-v)_X
      +\frac t2\norm{T^{{\star}}(\mu^{i+1}-\mu^i)}_W^2\\
 &\le\frac{\norm{\mu^i-v}_X^2-\norm{\mu^{i+1}-v}_X^2}{2\gamma},
 \end{aligned}
\end{equation}
where the last step uses $\gamma t\norm{T}^2\le1$.
Sum for $i=0,\ldots,j-1$, divide by $j$, and apply convexity of
the primal quadratic and $H(\mu^j)\le H(\mu^{i+1})$:
\[
 \frac1{2t}\norm{\overline y^{\,j}-q}_W^2
 +(T\overline y^{\,j},v)_X-R^*(v)+H(\mu^j)
 \le\frac{\norm{\mu^0-v}_X^2}{2\gamma j}
 \le\frac{(\norm{\mu^0}_X+L_R)^2}{2\gamma j}.
\]
Take the supremum over $v\in\dom R^*$ and use $R=R^{**}$.
This proves the first upper bound; $\norm{\mu^0}_X\le L_R$ gives
the second. Both objective values are finite, and weak duality gives
nonnegativity.
\end{proof}

The recursion $\overline y^{\,1}=y^0$,
$\overline y^{\,j+1}=(j\overline y^{\,j}+y^j)/(j+1)$ requires no
additional optimization solve. After $j$ dual updates, either candidate
may be tested. For the average, use $s=\overline y^{\,j}-w$ and the
full gap $\Psi_{q,t}(\overline y^{\,j})+H(\mu^j)$.
For example, with $W=X=\R$, $T=I$, $R(x)=|x|$, $q=2$, $t=1$,
$\gamma=1/2$, and $\mu^0=0$, we have $y^0=2$, $\mu^1=1$, and
$\overline y^{\,1}=2$. The pair $(2,1)$ has zero Fenchel defect but
$r=1$ and gap $1/2$. Thus both terms in \eqref{eq:gap_identity}
must be included.

\subsection{Sufficient inner iteration counts}
\begin{corollary}[A sufficient number of inner iterations]
\label{cor:inner_budget}
Under the hypotheses of Theorem~\ref{thm:inner}, let
$\varepsilon_{\mathrm{gap}}>0$. For the recovered pair,
$\mathcal G_{q,t}(y^j,\mu^j)\le\varepsilon_{\mathrm{gap}}$
whenever the integer $j$ satisfies
\begin{equation}\label{eq:absolute_inner_budget}
 j\ge\max\left\{2,
    1+\frac{4L_R\norm{\mu^0-\mu_\star}_X}
                 {\gamma\varepsilon_{\mathrm{gap}}}\right\}.
\end{equation}
For the averaged pair, it suffices that
\begin{equation}\label{eq:averaged_inner_budget}
 j\ge\max\left\{1,
          \frac{2L_R^2}{\gamma\varepsilon_{\mathrm{gap}}}\right\}.
\end{equation}
If $q=w-t\nabla f(w)$ for $w\in\dom\varphi$ and $p(w;t)\ne0$,
either bound with
$\varepsilon_{\mathrm{gap}}=\sigma^2\norm{p(w;t)}_W^2/(8t)$
guarantees the relative test \eqref{eq:certificate} for its corresponding
candidate step.
\end{corollary}
\begin{proof}
The absolute statements follow by substitution in
\eqref{eq:gap_rate} and \eqref{eq:averaged_gap_rate}, respectively.
For the relative claim, let $y$ be the candidate in question.
The appropriate absolute bound gives
$\mathcal G_{q,t}(y,\mu^j)\le\varepsilon_{\mathrm{gap}}$.
Since $y_\star=w+p(w;t)$, \eqref{eq:gap_distance} gives
\[
 \norm{y-y_\star}_W
 \le\sqrt{2t\,\mathcal G_{q,t}(y,\mu^j)}
 \le\sqrt{2t\varepsilon_{\mathrm{gap}}}
 =\frac\sigma2\norm{p(w;t)}_W.
\]
The reverse triangle inequality and $0<\sigma<1$ therefore imply
\[
 \norm{y-w}_W
 \ge\norm{p(w;t)}_W-\norm{y-y_\star}_W
 \ge\left(1-\frac\sigma2\right)\norm{p(w;t)}_W
 \ge\frac12\norm{p(w;t)}_W.
\]
Consequently,
$
 \mathcal G_{q,t}(y,\mu^j)
 \le\frac{\sigma^2\norm{p(w;t)}_W^2}{8t}
 \le\frac{\sigma^2}{2t}\norm{y-w}_W^2,
$
which is the relative test.
\end{proof}

The exact step appears only in this iteration estimate; the algorithm
tests the gap directly. Replacing $\norm{\mu^0-\mu_\star}_X$ by
$2L_R$ makes \eqref{eq:absolute_inner_budget} an a priori absolute
bound. The averaged bound already uses only known quantities.
At a stationary point and $t=\tau$, a gap at most
$\tau\varepsilon^2/8$ implies $U_\tau\le\varepsilon$ for either
construction. Indeed, $y_\star=w$ and \eqref{eq:gap_distance} imply
\[
 U_\tau(y,\mu;w)
 \le2\sqrt{\frac{2\mathcal G_{q,\tau}(y,\mu)}{\tau}}
 \le\varepsilon.
\]
The estimates remain valid if $T=0$, with any $\gamma>0$; if
$L_R=0$, the regularizer is constant and all generated gaps are zero.
Both constructions satisfy Assumption~\ref{ass:oracle}, so the outer
convergence and $O(\varepsilon^{-2})$ outer-trial bound apply.
The inner estimates concern a fixed proximal subproblem, not the total
work over all outer requests.

\subsection{An optional test for recovered pairs}
\label{sec:recovered_variant}
Exact recovery permits a less restrictive test; the numerical
experiments use the original tests. For $w\in\dom\varphi$, $t>0$,
and $q=w-t\nabla f(w)$, let $y=w+s=q-tT^{{\star}}\mu$ have finite
gap and consider
\begin{equation}\label{eq:recovered_variant}
 \mathcal G_{q,t}(w+s,\mu)
 \le\frac{\sigma^2}{t}\norm{s}_W^2.
\end{equation}
This permits twice the gap in \eqref{eq:certificate}, yet gives the
same conclusions in Theorem~\ref{thm:certificate}:
\eqref{eq:recovered_gap_distance} yields
$\norm{s-p(w;t)}_W\le\sigma\norm s_W$. The proof of
Theorem~\ref{thm:certificate}, with $r=0$, gives
\[
 \ip{\nabla f(w)}s+\varphi(w+s)-\varphi(w)
 \le-\frac{1-\sigma^2}{t}\norm s_W^2
 \le-\frac{1-\sigma}{t}\norm s_W^2.
\]
The norm comparisons and quadratic-model decrease follow as before.

At the reference parameter $t=\tau$, pair this test with
\begin{equation}\label{eq:recovered_variant_upper}
 \widehat U_\tau(y,\mu;w)
 \defeq\frac{\norm{y-w}_W}{\tau}
   +\sqrt{\frac{\mathcal G_{q,\tau}(y,\mu)}{\tau}}.
\end{equation}
The recovered error bound gives $\chi(w)\le\widehat U_\tau$.
Under \eqref{eq:recovered_variant},
\[
 \widehat U_\tau\le(1+\sigma)\norm s_W/\tau,
 \qquad \chi(w)\ge(1-\sigma)\norm s_W/\tau.
\]
Failure of $\widehat U_\tau\le\varepsilon$ followed by success of
the relative test therefore gives $\chi(w)>c_\sigma\varepsilon$.
Vanishing gaps give finite stopping by the argument of
Proposition~\ref{prop:screening}. Since the same distance, decrease,
and continuation bounds hold, the Cauchy search, outer convergence,
and outer-trial estimate retain their constants. Averaged candidates
continue to use the original tests.

\section{{A semilinear control problem with total variation}}
\label{sec:pde}

Let $\Omega\subset\R^N$ be a bounded Lipschitz domain with
$1\le N\le4$, let $u_d\in L^2(\Omega)$, and fix
$\alpha,\beta>0$. For a control $z$, the state $u=u(z)$ solves
\begin{equation}\label{eq:pde_state}
 -\Delta u+u^3=z\quad\text{in }\Omega,
 \qquad u=0\quad\text{on }\partial\Omega.
\end{equation}
We consider
\begin{equation}\label{eq:pde_objective}
 \min_{z\in H^1(\Omega)} F(z),\qquad
 F(z)\defeq\underbrace{\frac12\norm{u(z)-u_d}_{L^2}^2
                +\frac\alpha2\norm{z}_{L^2}^2}_{f(z)}
                +\beta\int_\Omega|\nabla z|\dd x.
\end{equation}
All norms and integrals in this section are over $\Omega$ unless
otherwise indicated. The control has no prescribed boundary trace.
Semilinear elliptic control with ${\mathrm{BV}}$ costs is studied, for example,
by \citet{CasasKunisch2019}.

The algorithm uses the full $H^1$ metric,
\begin{equation}\label{eq:pde_metric}
 W=H^1(\Omega),\qquad
 \ip{z}{d}\defeq\int_\Omega zd+\nabla z\cdot\nabla d\dd x.
\end{equation}
This metric defines the gradient, proximal subproblems, and 
trust-region radii; the quadratic control cost remains an $L^2$ cost. In the
abstract notation, the outer variable $w$ is the control $z$ and a
primal proximal point $y$ is another control, whereas $u$ denotes a state.

\subsection{Smooth part and the algorithmic assumptions}

\begin{proposition}[State map and global smoothness]
\label{prop:pde_smoothness}
For every $z\in L^2(\Omega)$, \eqref{eq:pde_state} has a unique
weak solution in $H_0^1(\Omega)$. The state map
$u:L^2(\Omega)\to H_0^1(\Omega)$ is {twice continuously
differentiable}. The reduced functional $f$ is {twice
continuously differentiable} on $L^2(\Omega)$, with a
globally bounded Hessian. Consequently, its Hilbert gradient in
\eqref{eq:pde_metric} is globally Lipschitz on $W$.
\end{proposition}

The proof and an explicit control-independent bound are given in
Appendix~\ref{app:pde_smoothness}. For implementation, let the adjoint
$\lambda$ and the linearized state $v_d$, for $d\in L^2(\Omega)$,
solve
\begin{equation}\label{eq:pde_adjoint_linearized}
 \begin{aligned}
 -\Delta\lambda+3u^2\lambda&=u-u_d,\\
 -\Delta v_d+3u^2v_d&=d,
 \end{aligned}
 \qquad \lambda,v_d\in H_0^1(\Omega).
\end{equation}
Then
\begin{equation}\label{eq:pde_derivatives}
 \begin{aligned}
 f'(z)[d]&=(\lambda+\alpha z,d)_{L^2},\\
 f''(z)[d,e]&=\alpha(d,e)_{L^2}
           +\int_\Omega(1-6u\lambda)v_dv_e\dd x.
 \end{aligned}
\end{equation}
The coefficient $1-6u\lambda$ can have either sign. The gradient
used by the algorithm is determined by the Riesz equation
\begin{equation}\label{eq:pde_hilbert_gradient}
 \ip{\nabla f(z)}{d}=(\lambda+\alpha z,d)_{L^2}
 \qquad(d\in H^1(\Omega)).
\end{equation}

For the nonsmooth term, take the Hilbert space
\begin{equation}\label{eq:pde_composition}
 X=L^2(\Omega;\R^N),\qquad Tz=\nabla z,\qquad
 R(v)\defeq\beta\int_\Omega|v|\dd x.
\end{equation}
Then $\norm{T}\le1$ and $R$ is finite, convex, and globally Lipschitz,
since
\[
 |R(v)-R(\widetilde v)|
 \le\beta|\Omega|^{1/2}\norm{v-\widetilde v}_X.
\]
Under the Hilbert identification of $X$ with its dual,
\begin{equation}\label{eq:pde_conjugate}
 R^*(\mu)=\begin{cases}
 0,&|\mu(x)|\le\beta\quad\text{for almost every }x,\\
 +\infty,&\text{otherwise}.
 \end{cases}
\end{equation}
Thus, for every
$\gamma>0$,
\begin{equation}\label{eq:pde_dual_projection}
 \prox_{\gamma R^*}(\mu)(x)
 =\frac{\mu(x)}{\max\{1,|\mu(x)|/\beta\}}.
\end{equation}
The feasible initial multiplier $\mu^0=0$ and the conjugate values
are explicit. The Hilbert adjoint of $T$ satisfies
\begin{equation}\label{eq:pde_adjoint_operator}
 \ip{T^{{\star}}\mu}{d}
 =\int_\Omega\mu\cdot\nabla d\dd x
 \qquad(d\in H^1(\Omega)).
\end{equation}
This variational definition includes the boundary contribution and
does not require a normal trace of an arbitrary $\mu\in X$.

Proposition~\ref{prop:pde_smoothness}, $F\ge0$, and
\eqref{eq:pde_composition} verify Assumption~\ref{ass:basic}.
Equations~\eqref{eq:pde_conjugate}--\eqref{eq:pde_dual_projection}
provide the evaluations in Assumption~\ref{ass:lipschitz_R}; 
the dual proximal-gradient iteration may use $\gamma=1/t$ because 
$\norm{T}\le1$. The constructive gap results of Section~\ref{sec:inner} 
therefore apply, under the stated assumptions on exact evaluations.

\begin{remark}[Banach-space formulation]
The same problem also admits the Banach-space factorization
$X_1=L^1(\Omega;\R^N)$, $T_1z=\nabla z$, and
$R_1(v)\defeq\beta\norm{v}_{L^1}$, with $W$ unchanged.
Here $X_1^*=L^\infty$ and $R_1^*$ is the indicator of
$K\defeq\{\mu\in L^\infty:|\mu|\le\beta\text{ a.e.}\}$.
Since $|\Omega|<\infty$, this is also the feasible multiplier set
in the $L^2$ formulation above. The primal objectives, dual objectives
on $K$, and gap certificates coincide. Thus the $L^2$ projected-dual
iteration supplies the inner procedure in Assumption~\ref{ass:oracle}
for this Banach-space formulation.
\end{remark}

\subsection{\texorpdfstring{$H^1$}{H1} control}

Write $|Dz|(\Omega)$ for the Euclidean total variation of a
${\mathrm{BV}}$ function. For $z\in H^1(\Omega)$ it equals
$\int_\Omega|\nabla z|\dd x$.

\begin{proposition}[$H^1$ control]
\label{prop:pde_existence}
The extension of \eqref{eq:pde_objective} to
${\mathrm{BV}}(\Omega)\cap L^2(\Omega)$, with regularizer
$\beta|Dz|(\Omega)$, has a global minimizer. If $\Omega$ is convex,
every such minimizer
$\bar z$ belongs to $H^1(\Omega)$ and its adjoint satisfies
\begin{equation}\label{eq:pde_optimal_control_regularity}
 \alpha\norm{\nabla\bar z}_{L^2}
 \le\norm{\nabla\bar\lambda}_{L^2}.
\end{equation}
In this case, \eqref{eq:pde_objective} attains its minimum in $W$.
\end{proposition}

Appendix~\ref{app:pde_existence} proves this result by the direct
method in ${\mathrm{BV}}\cap L^2$, followed by the convex-domain regularity
theorem of \citet{BouchutCarstensenErn2025}.

\section{Numerical experiments}
\label{sec:pde_numerics}

We apply Algorithm~\ref{alg:outer} to \eqref{eq:pde_objective} with
$\Omega=(0,1)^2$ and
\begin{equation}\label{eq:num_data}
 u_d(x)=\sin(2\pi x_1)\sin(2\pi x_2),\qquad
 \alpha=10^{-6},\qquad\beta=10^{-4}.
\end{equation}
We examine iteration counts and computational cost under mesh refinement.
{We divide} the square into $n\times n$ squares and split each along its
southwest--northeast diagonal. Continuous piecewise affine functions
approximate the control in $W_h\subset H^1(\Omega)$ and the state and
adjoint in $V_h=W_h\cap H_0^1(\Omega)$. There are $(n+1)^2$ control
unknowns. The metric is the full $H^1$ inner product
\eqref{eq:pde_metric}. On each triangle, we use a positive seven-point
quadrature rule $\mathcal Q_h$, exact through total degree five. This integrates the
quartic polynomial terms in the state equation and its derivatives 
exactly. The same rule defines the discrete tracking functional 
and its derivatives consistently. 
With $a_K\defeq|K|$,
the discrete objective is
\begin{equation}\label{eq:num_discrete_objective}
 F_h(z_h)\defeq\tfrac12\mathcal Q_h\bigl((u_h(z_h)-u_d)^2\bigr)
       +\tfrac\alpha2\norm{z_h}_{L^2}^2
       +\beta\sum_K a_K|\nabla z_h|_K.
\end{equation}
The last term integrates the {isotropic total variation} exactly, without
smoothing. State, adjoint, and sensitivity equations are the Galerkin
forms of \eqref{eq:pde_state} and \eqref{eq:pde_adjoint_linearized},
using the same quadrature.

\subsection{Implementation}
\label{sec:num_implementation}

Every run starts independently from $z_h^0=0$, $\mu^0=0$, with
\begin{equation}\label{eq:num_parameters}
 \begin{gathered}
 \varepsilon=10^{-6},\quad\tau=1,\quad\sigma=\tfrac12,
 \quad\Delta_0=\Delta_{\max}=10^3,\quad\nu=1,\\
 \zeta=\gamma_c=\tfrac12,\quad\gamma_e=2,
 \quad\eta_1=0.1,\quad\eta_2=0.75.
 \end{gathered}
\end{equation}
The implementation follows four steps in Algorithm~\ref{alg:outer}.
\begin{enumerate}
\item \emph{{Derivatives and stationarity.}}
{We solve} the state equation by damped Newton, starting trial states from
the current state. {We factor} the Jacobian at the computed state and reuse
it for the adjoint and subsequent linearized solves. The linear systems
are solved by sparse SuperLU \citep{Li2005SuperLU}.
{We measure} the state residual in $V_h^*$, where $V_h$ carries the norm
$\norm{\nabla\cdot}_{L^2}$, and use the tolerance
$10^{-12}+2\cdot10^{-13}\max\{1,\norm{z_h}_{V_h^*}\}$.
{We apply} Algorithm~\ref{alg:reference} to the proximal problem
\eqref{eq:exact_prox} at $t=\tau$, first testing $y=z_h^k$ with
the available feasible multiplier. {We generate} further pairs by accelerated
dual proximal gradient \citep{BeckTeboulle2009} with $\gamma=1/t$;
if needed, {we switch} after 1000 updates to the ordinary iteration
\eqref{eq:dual_iteration}.
At each tested pair, {we terminate} the outer algorithm if
$U_\tau\le\varepsilon$; otherwise {we end} the proximal call when
\eqref{eq:certificate} holds and continue the outer iteration.
Tests use projected multipliers satisfying $|\mu_K|\le\beta$ on
every triangle and both terms of \eqref{eq:gap_identity}.
Dual updates and certificate evaluations use solves with the fixed
$H^1$ Riesz matrix, factored once per mesh, and elementwise ball
projections (ensuring dual feasibility); no further state or adjoint solves are needed.
All reported proximal calls terminate before the switch.
\item \emph{{Cauchy step.}}
{We use} the projected curvature defined below and apply
Algorithm~\ref{alg:cauchy}, starting at
$t=\min\{\tau,(1-\sigma)/(1+b_k)\}$ and halving $t$ until the
relatively certified step lies within $\Delta_k$. {We denote} the returned
step by $s_k^{\rm C}$.
\item \emph{Optional model decrease.}
{We approximately minimize} $m_k(d)$ without its trust-region constraint.
{We replace} TV by $\beta\sum_K a_K\xi_K$ with
$\xi_K\ge|\nabla(z_h^k+d)|_K$, giving an equivalent quadratic conic
problem \citep{AlizadehGoldfarb2003}. Our code assembles it for the
interior-point (IP) solver Clarabel~0.11.1 \citep{GoulartChen2026Clarabel}.
{We scale} a finite conic candidate $d$ to
$\widehat d=\min\{1,\Delta_k/\norm{d}_W\}d$ (scale one at $d=0$).
{We evaluate} $B_k\widehat d$ and the elementwise TV to recompute
$\pred_k(\widehat d)$ from \eqref{eq:model}.
{We use} $\widehat d$ if $\pred_k(\widehat d)\ge\pred_k(s_k^{\rm C})$;
otherwise {we retain} $s_k^{\rm C}$, also used if no finite candidate is returned.
Thus \eqref{eq:improved_step} holds with $\nu=1$.
When the conic candidate is selected, {we undo} the objective and element
scalings of its TV multiplier, project onto $|\mu_K|\le\beta$, and
reuse it in subsequent proximal calls.
\item \emph{{Trial evaluation and acceptance.}}
{We solve} the nonlinear state equation at $z_h^k+d_k$ and evaluate
$F_h(z_h^k+d_k)$. {We use}
$\rho_k=[F_h(z_h^k)-F_h(z_h^k+d_k)]/\pred_k(d_k)$ in the acceptance
and radius updates of Algorithm~\ref{alg:outer}.
{We reuse} the trial state on acceptance; on rejection, {we retain} the current
state, adjoint, and model.
\end{enumerate}

For the curvature in \eqref{eq:model}, {we set}
$\omega_k\defeq(1-6u_k\lambda_k)_+$ at quadrature points, using the current
state and adjoint. With $v_d$ the discrete linearized state in direction
$d$, the full model is
\begin{equation}\label{eq:num_full_model}
 (B_k^{\rm full}d,e)_W
   \defeq\alpha(d,e)_{L^2}+\mathcal Q_h(\omega_kv_dv_e).
\end{equation}
Replacing negative curvature weights by zero makes the optional model
convex; the objective and its gradient remain unchanged.
For $p=5$, {we form} the nodal interpolants of
$\sin(i\pi x_1)\sin(j\pi x_2)$, $1\le i,j\le p$, and diagonalize their
Gram matrix in $\mathcal Q_h(\omega_k\,\cdot\,\cdot)$.
{We retain} eigenvector combinations with eigenvalues
exceeding $10^{-12}$ times the largest eigenvalue, and normalize them to
obtain $\theta_1,\ldots,\theta_r$ with
$\mathcal Q_h(\omega_k\theta_i\theta_j)=\delta_{ij}$. {We define}
\begin{equation}\label{eq:num_projected_model}
 \begin{aligned}
 \Pi_kv&\defeq\sum_{j=1}^r\mathcal Q_h(\omega_kv\theta_j)\theta_j,\\
 (B_k^{\rm proj}d,e)_W
   &\defeq\alpha(d,e)_{L^2}
          +\mathcal Q_h\bigl(\omega_k(\Pi_kv_d)(\Pi_kv_e)\bigr).
 \end{aligned}
\end{equation}
Here $\Pi_k$ is the unique weighted least-squares projection onto
$S_k\defeq\operatorname{span}\{\theta_1,\ldots,\theta_r\}$, with $\Pi_k=0$
if $r=0$. Weighted orthogonality gives
\[
 \bigl((B_k^{\rm full}-B_k^{\rm proj})d,d\bigr)_W
 =\mathcal Q_h\bigl(\omega_k(v_d-\Pi_kv_d)^2\bigr)\ge0,
\]
so $0\le B_k^{\rm proj}\le B_k^{\rm full}$ in quadratic-form order.
The finite element spaces are unchanged. 

For exact discrete state and adjoint solutions, the energy estimates
in Appendix~\ref{app:pde_smoothness} hold with $\norm{u_d}_{L^2}^2$
replaced by $\mathcal Q_h(u_d^2)$: polynomial products are integrated
exactly, and the target load is bounded by quadrature Cauchy--Schwarz.
These estimates, H\"older's inequality for $\mathcal Q_h$, and
$\omega_k\le1+6|u_k\lambda_k|$ give
\begin{equation}\label{eq:num_model_bound}
 \norm{B_k}\le b_k\defeq \min\{\alpha+C_P^4\max\omega_k,\overline b_h\},
 \qquad C_P=(\sqrt2\pi)^{-1},\quad C_4^2=2C_P.
\end{equation}
Ladyzhenskaya's inequality gives $C_4^2=2C_P$.
The maximum is over quadrature points, and $\overline b_h$ is
\eqref{eq:pde_global_lipschitz_constant} with
$\norm{u_d}_{L^2}^2$ replaced by $\mathcal Q_h(u_d^2)$.
Algorithm~\ref{alg:cauchy} needs only an upper bound on $\norm{B_k}$;
sharp constants are unnecessary.
The invertible state Jacobian gives a smooth discrete state map.
Applying the same estimates to the signed Hessian weight
$1-6u_k\lambda_k$ bounds $f_h''$ by $\overline b_h$, so the discrete
smooth gradient is globally Lipschitz. Since $\mathcal Q_h(u_d^2)\le1$,
these bounds are uniform in the mesh.

The optional conic solve is limited to 100 IP iterations.
Early stopping is permitted after at least three iterations when
the reported gap, in objective units, is at most $0.1$ times a
positive reported model reduction and both residuals are at most
$\min\{10^{-5},10^{-3}U_\tau^2\}$.

\subsection{Mesh refinement}
\label{sec:num_results}

Table~\ref{tab:num_mesh} reports mesh refinement with 25 curvature
modes. Counts are totals over optimization: outer trials,
projected-dual updates, conic IP iterations, and nonlinear state
Newton corrections. Both outer trials are accepted on every mesh.
The reported $U_\tau$ uses the computed state and adjoint and
measures stationarity of the discrete problem.
Times cover optimization only, excluding mesh assembly and the
fixed initial Riesz and residual factorizations. We use Python~3.12,
SciPy~1.18, and Clarabel~0.11.1 on macOS arm64, with one requested
numerical thread.

\begin{table}[!htbp]
 \centering\small\setlength{\tabcolsep}{4pt}
 \input{tables/table_mesh_compact.tex}
 \caption{Mesh refinement with 25 curvature modes.
 Dual counts include accelerated updates; Newton counts are state
 corrections. Each run calls the proximal routine five times and the optional
conic solver twice; a proximal call may require no dual updates.}
 \label{tab:num_mesh}
\end{table}

\begin{figure}[h!]
 \centering
 \includegraphics[width=\textwidth]{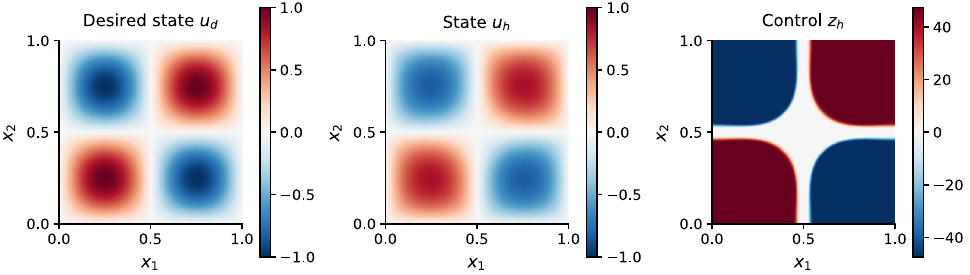}
 \caption{Desired state, computed state, and control for $n=512$ with
 25 curvature modes. The states share a color scale; the control uses
 its own scale. }
 \label{fig:num_fields}
\end{figure}

Outer and state Newton counts remain constant over the tested meshes,
while IP totals vary from 28 to 52. This is numerical evidence of
mesh-independent outer iteration counts at this fixed tolerance.
Runtime increases with the size of the linear systems.
Figure~\ref{fig:num_fields} shows the computed fields on the
finest mesh, with 263,169 control unknowns.

\FloatBarrier
\paragraph{{Acknowledgement.}}
{The author thanks Johannes Milz for carefully reading the manuscript
and for helpful comments on its presentation.}

\appendix

\section{PDE analysis for the semilinear control problem}
\label{app:pde}

This appendix proves the results stated in Section~\ref{sec:pde}.
Set $V\defeq H_0^1(\Omega)$ with norm
$\norm{v}_V\defeq\norm{\nabla v}_{L^2}$ and let $V^*=H^{-1}(\Omega)$
carry the corresponding dual norm. Fix constants $C_P,C_4$ such that
\begin{equation}\label{eq:pde_embedding_constants}
 \norm{v}_{L^2}\le C_P\norm{\nabla v}_{L^2},\qquad
 \norm{v}_{L^4}\le C_4\norm{\nabla v}_{L^2}
 \qquad(v\in V).
\end{equation}
The assumption $N\le4$ ensures the continuous embedding
$V\hookrightarrow L^4(\Omega)$.

\subsection{State map, derivatives, and a global Hessian bound}
\label{app:pde_smoothness}

\begin{proof}[Proof of Proposition~\ref{prop:pde_smoothness}]
The embedding $V\hookrightarrow L^4(\Omega)$
\citep[Theorem~7.1, p.~355]{Troeltzsch2010}
makes the cubic a continuous polynomial from $V$ to $V^*$.
Thus $A(u)\defeq-\Delta u+u^3$ is smooth, and
\[
 \pair{A(u)-A(v)}{u-v}\ge\norm{u-v}_V^2.
\]
Since $A$ is smooth, it is hemicontinuous; since $A(0)=0$,
strong monotonicity also implies coercivity. The
monotone-operator theorem
\citep[Theorem~4.1, p.~185]{Troeltzsch2010}
therefore shows that $A:V\to V^*$ is bijective. Strong monotonicity gives
$\norm{u(z_1)-u(z_2)}_V
 \le C_P\norm{z_1-z_2}_{L^2}$ for $z_1,z_2\in L^2(\Omega)$.

For every $u\in V$, the derivative
$J_u=A'(u)=-\Delta+3u^2$ has a bounded coercive bilinear form,
so it is an isomorphism by Lax--Milgram.
The inverse function theorem gives a smooth local inverse at
every point. Since $A$ is bijective, these local inverses agree
with $A^{-1}$, which is therefore smooth on $V^*$.
Restricting to $L^2(\Omega)\hookrightarrow V^*$ gives the
asserted differentiability of the state map.
For $v_d=u'(z)d$ and $w_{d,e}=u''(z)[d,e]$, differentiation yields
\[
 J_uv_d=d,\qquad J_uw_{d,e}=-6uv_dv_e
 \quad\text{in }V^*.
\]
Testing these equations with
$\lambda=J_u^{-1}(u-u_d)$ and using symmetry and the chain rule
gives $f\in C^2(L^2(\Omega);\R)$ and
\eqref{eq:pde_derivatives}. All products are justified by
$V\hookrightarrow L^4(\Omega)$.

For the global bound, use the constants $C_P,C_4$ from
\eqref{eq:pde_embedding_constants}.
Testing the linearized equation with $v_d$ gives
\begin{equation}\label{eq:pde_linearized_bound}
 \norm{\nabla v_d}_{L^2}\le C_P\norm{d}_{L^2}.
\end{equation}
Testing the adjoint equation with $\lambda$ yields
\[
 \norm{\nabla\lambda}_{L^2}^2+3\norm{u\lambda}_{L^2}^2
 \le |\Omega|^{1/2}\norm{u\lambda}_{L^2}
      +C_P\norm{u_d}_{L^2}\norm{\nabla\lambda}_{L^2}.
\]
Young's inequality therefore gives the control-independent bound
\begin{equation}\label{eq:pde_uniform_adjoint}
 \norm{\nabla\lambda}_{L^2}^2+3\norm{u\lambda}_{L^2}^2
 \le C_P^2\norm{u_d}_{L^2}^2+\frac{|\Omega|}{3}.
\end{equation}
Since $\norm{v_d}_{L^2}\le C_P^2\norm{d}_{L^2}$ and
$\norm{v_d}_{L^4}\le C_4C_P\norm{d}_{L^2}$, H\"older's inequality gives
\[
 |f''(z)[d,e]|
 \le\bigl(\alpha+C_P^4+6C_4^2C_P^2\norm{u\lambda}_{L^2}\bigr)
       \norm{d}_{L^2}\norm{e}_{L^2}.
\]
Thus $|f''(z)[d,e]|\le L_f\norm{d}_{L^2}\norm{e}_{L^2}$ with
\begin{equation}\label{eq:pde_global_lipschitz_constant}
 L_f\defeq\alpha+C_P^4+2\sqrt3\,C_4^2C_P^2
       \left(C_P^2\norm{u_d}_{L^2}^2
                         +\frac{|\Omega|}{3}\right)^{1/2}.
\end{equation}
Since $\norm{d}_{L^2}\le\norm{d}_W$, the same constant bounds
the Hessian on $W$. Integration along line segments and the
Riesz identity~\eqref{eq:pde_hilbert_gradient} give
\[
 \norm{\nabla f(z_1)-\nabla f(z_2)}_W
 \le L_f\norm{z_1-z_2}_W
 \qquad(z_1,z_2\in W).
\]
\end{proof}

\subsection{Existence and regularity of optimal controls}
\label{app:pde_existence}

\begin{proof}[Proof of Proposition~\ref{prop:pde_existence}]
Extend the objective to $L^2$ by setting it equal to $+\infty$
outside ${\mathrm{BV}}\cap L^2$. Its infimum is finite, since the objective is
nonnegative and finite at zero. A minimizing sequence $(z_n)$ is
bounded in $L^2$ by the quadratic cost and in ${\mathrm{BV}}$ by the total
variation term and $\norm{z_n}_{L^1}\le|\Omega|^{1/2}\norm{z_n}_{L^2}$.
Compactness in ${\mathrm{BV}}$ and weak compactness in $L^2$ give, after extraction,
\[
 z_n\to\bar z\quad\text{in }L^1(\Omega),\qquad
 z_n\rightharpoonup\bar z\quad\text{in }L^2(\Omega),
\]
where testing against bounded functions identifies the two limits.
Lower semicontinuity of total variation gives $\bar z\in {\mathrm{BV}}\cap L^2$.

The inclusion $L^2\hookrightarrow V^*$ is compact, being the adjoint
of the compact inclusion $V\hookrightarrow L^2$. Hence
$z_n\to\bar z$ in $V^*$, and strong monotonicity of the state operator gives
\[
 \norm{u(z_n)-u(\bar z)}_V\le\norm{z_n-\bar z}_{V^*}\to0.
\]
Thus the tracking term converges. Lower semicontinuity of the
quadratic control cost and total variation proves minimality of $\bar z$.

Fix any minimizer $\bar z$ and $v\in {\mathrm{BV}}\cap L^2$. Minimality along
the segment $\bar z+\theta(v-\bar z)$ and convexity of total variation give
\[
 0\le f(\bar z+\theta(v-\bar z))-f(\bar z)
       +\theta\beta\bigl(|Dv|(\Omega)-|D\bar z|(\Omega)\bigr)
 \qquad(0<\theta<1).
\]
Dividing by $\theta$ and letting $\theta\downarrow0$ yields
\begin{equation}\label{eq:pde_relaxed_optimality}
 (\bar\lambda+\alpha\bar z,v-\bar z)_{L^2}
       +\beta\bigl(|Dv|(\Omega)-|D\bar z|(\Omega)\bigr)\ge0.
\end{equation}
It follows that $\bar z$ uniquely minimizes
\begin{equation}\label{eq:pde_frozen_denoising}
 v\longmapsto\frac\alpha2
        \norm{v+\bar\lambda/\alpha}_{L^2}^2
       +\beta|Dv|(\Omega),\qquad v\in {\mathrm{BV}}(\Omega)\cap L^2(\Omega).
\end{equation}
Indeed, the difference of its values at $v$ and $\bar z$ is the
left-hand side of \eqref{eq:pde_relaxed_optimality} plus
$\alpha\norm{v-\bar z}_{L^2}^2/2$. This uniqueness concerns the
convex problem with $\bar\lambda$ fixed.

Suppose now that $\Omega$ is convex. In its homogeneous Neumann case
(zero boundary penalty in their equation~(1.5)),
\citet[Theorem~1.1]{BouchutCarstensenErn2025} states that the minimizer of
\[
 \frac a2\norm{v}_{L^2}^2+|Dv|(\Omega)-(g,v)_{L^2},
 \qquad v\in {\mathrm{BV}}(\Omega)\cap L^2(\Omega),
\]
on a bounded Lipschitz convex domain belongs to $H^1$ and satisfies
$a\norm{\nabla v}_{L^2}\le\norm{\nabla g}_{L^2}$ for $a>0$ and
$g\in H^1$. Divide \eqref{eq:pde_frozen_denoising} by $\beta$ and
omit its additive constant: $a=\alpha/\beta$ and
$g=-\bar\lambda/\beta\in H_0^1(\Omega)$ satisfy these hypotheses.
Thus $\bar z\in H^1$ and \eqref{eq:pde_optimal_control_regularity}
holds. Since the objectives agree on $H^1$, every relaxed minimizer
also solves \eqref{eq:pde_objective}.
\end{proof}

\FloatBarrier
\bibliographystyle{abbrvnat}
\bibliography{references}

\end{document}

%% file: tables/table_mesh_compact.tex
\begin{tabular}{rrrrrrrr}
\toprule
$n$ & Control dofs & Outer & Dual & IP & \shortstack{State\\ Newton} & Time (s) & $U_\tau$\\
\midrule
32 & 1089 & 2 & 6 & 28 & 5 & 0.13 & $6.45\times10^{-7}$\\
64 & 4225 & 2 & 2 & 36 & 5 & 0.66 & $4.60\times10^{-7}$\\
128 & 16641 & 2 & 4 & 47 & 5 & 3.89 & $4.70\times10^{-7}$\\
256 & 66049 & 2 & 5 & 52 & 5 & 22.41 & $4.92\times10^{-7}$\\
512 & 263169 & 2 & 3 & 47 & 5 & 116.22 & $5.03\times10^{-7}$\\
\bottomrule
\end{tabular}